\documentclass{amsart}

\usepackage{xcolor}
\usepackage{amsfonts,mathrsfs,amssymb,amsthm,mathptm}
\usepackage{amsmath,amscd}
\usepackage{amsmath}
\usepackage{mathptm,pslatex}
\usepackage{multirow}
\usepackage{color}
\usepackage[all]{xy}
\usepackage{url}
\usepackage{cite}
\usepackage{exscale}
\usepackage{relsize}
\usepackage{bbm}
\usepackage{enumerate}
\usepackage{paralist}
\usepackage{bm}
\usepackage{hyperref}
\hypersetup{colorlinks=true,citecolor=red,linkcolor=blue,CJKbookmarks=true}
\usepackage{amsrefs}
\usepackage{tikz-cd}
\usepackage{graphicx}
\xyoption{all}

\usepackage{paralist}
\newtheorem{Thm}{Theorem}[section]
\newtheorem{Def}[Thm]{Definition}
\newtheorem{Lem}[Thm]{Lemma}
\newtheorem{Cor}[Thm]{Corollary}
\newtheorem{Prop}[Thm]{Proposition}
\newtheorem{Eg}[Thm]{Example}
\newtheorem{Rem}[Thm]{Remark}

\newcommand{\Hom}{{\rm{Hom}}}

\newcommand{\Ext}{{\rm{Ext}}}
\newcommand{\End}{{\rm{End}}}
\newcommand{\modn}{{\rm{mod}}}
\newcommand{\Mod}{{\rm{Mod}}}

\renewcommand{\ker}{{\rm{Ker}}}

\newcommand{\soc}{{\rm{soc}}}
\renewcommand{\top}{{\rm{top}}}

\newcommand{\proj}{{\rm{proj }}}

\newcommand{\gldim}{{\rm{gldim}}}
\newcommand{\add}{{\rm{add}}}

\newcommand{\perf}{{\rm{perf}}}
\newcommand{\tri}{{\rm{tri}}}

\newcommand{\projdim}{{\rm{proj.dim}}}
\numberwithin{equation}{section}
\allowdisplaybreaks[4]

\begin{document}
\title{Singularity Categories of Higher Nakayama Algebras}

\author{Wei Xing}
\address{Uppsala University, Uppsala 75106, Sweden}
\curraddr{}
\email{wei.xing@math.uu.se}
\thanks{}


\date{}
\dedicatory{}

\renewcommand{\thefootnote}{\alph{footnote}}
\setcounter{footnote}{-1} 


\begin{abstract}
  For a higher Nakayama algebra $A$
  in the sense of Jasso-K\"{u}lshammer, 
  we give an alternative approach to show that the singularity category of $A$ is triangulated equivalent to the stable module category of a self-injective higher Nakayama algebra, which was proved by McMahon by contraction along fabric idempotent ideals. This also generalizes a similar result for usual Nakayama algebras due to Chen-Ye and Shen. Our proof relies on the existence of $d\mathbb{Z}$-cluster tilting subcategories in the module category of $A$ and the result of Kvamme that each $d\mathbb{Z}$-cluster tilting subcategory of $A$ induces a $d\mathbb{Z}$-cluster tilting subcategory in its singularity category. Moreover, our result provides many concrete examples of the derived Auslander-Iyama correspondence introduced by Jasso-Muro, namely, 
  for each higher Nakayama algebra we calculate the endomorphism algebra of its distinguished $d\mathbb{Z}$-cluster tilting object inside its singularity category, which can be used to recognize such singularity categories among algebraic triangulated categories.
\end{abstract}

\maketitle

\tableofcontents

\section{Introduction}\label{sect1}

Auslander-Reiten theory is a fundamental tool to study representation theory from a homological point of view.
A generalization of this theory, called higher
Auslander-Reiten theory,
was introduced by Iyama \cite{Iya07a, Iya07b, Iya11}.
In this theory, 
the object of study is some category $\mathcal{A}$,
usually the module category of a finite dimensional algebra or its bounded derived category, 
equipped with a $d$-cluster tilting subcategory $\mathcal{M} \subset \mathcal{A}$,
possibly with some additional property.
Depending on different settings, $d$-cluster tilting subcategories give rise to higher notions in homological algebra.
For instance, if $\mathcal{A}$ is abelian and $\mathcal{M}$ is $d$-cluster tilting,
then $\mathcal{M}$ is a $d$-abelian category in the sense of Jasso \cite{Jas16}.
If $\mathcal{A}$ is triangulated and $\mathcal{M}$ is $d$-cluster tilting with the additional property that $\mathcal{M}$ is closed under the $d$-fold suspension functor then $\mathcal{M}$ is $(d+2)$-angulated in the sense of Geiss-Keller-Oppermann \cite{GKO13}.

Let $A$ be a finite dimensional algebra, $\modn A$ be the category of finitely generated right $A$-modules and $D^b(\modn A)$ be the bounded derived category of $\modn A$.
Assume $\mathcal{M}$ is a $d$-cluster tilting subcategory of $\modn A$. 
A natural question is whether we can construct a $d$-cluster tilting subcategory $\mathcal{U}$ of some triangulated category related to $A$ out of $\mathcal{M}$.
If $A$ has global dimension $d$,
then the subcategory
\begin{equation*}
    \mathcal{U} = \add \{M[di]\in D^b(\modn A) \mid M\in \mathcal{M} \text{ and } i\in \mathbb{Z}\}
\end{equation*}
is $d$-cluster tilting inside $D^b(\modn A)$, see \cite{Iya11}.
In general,
if we drop the assumption that $A$ has global dimension $d$,
there is no known cluster tilting subcategories inside $D^b(\modn A)$.
As shown in \cite{JK16}, 
the naive approach doesn't necessarily give a $d$-cluster tilting subcategory in $D^b(\modn A)$.
If $A$ is self-injective, its stable module category $\underline{\modn} A$ has a triangulated structure, see \cite{Ha88}.
Then $\mathcal{U} = \underline{\mathcal{M}}$ is $d$-cluster tilting. Indeed, all $d$-cluster tilting subcategories of $\underline{\modn }A$ arise in this way.

With the stronger assumption that $\mathcal{M}$ is $d\mathbb{Z}$-cluster tilting \cite{IJ17},
Kvamme \cite{Kv21} showed that,
every $d\mathbb{Z}$-cluster tilting subcategory of $\modn A$ gives rise to a $d\mathbb{Z}$-cluster tilting subcategory of the singularity category $D_{sg}(A)$.
Note that if $\gldim A = d$, then $D_{sg}(A) = 0$ and any
$d$-cluster tilting subcategory is trivially $d\mathbb{Z}$-cluster tilting.
If $A$ is self-injective, then
 $D_{sg}(A) = \underline{\modn} A$ and there is a bijective correspondence between $d\mathbb{Z}$-cluster tilting subcategories in $\modn A$ and $\underline{\modn }A$.

However, the existance of $d\mathbb{Z}$-cluster tilting imposes a strong restriction on $A$. Nakayama algebras provide some examples of such algebras.
Recently, Herschend-Kvamme-Vaso \cite{HKV22} gave a complete description of $d\mathbb{Z}$-cluster tilting subcategories of Nakayama algebras.
Another typical class of such algebras is higher Nakayama algebras constructed by Jasso-K\"{u}lshammer \cite{JK19}.
As a generalization of usual Nakayama algebras,
higher Nakayama algebras admit complex homological structures while being convenient to compute combinatorically.
For this reason, we restrict our objects of study to higher Nakayama algebras.

The singularity category of an algebra was introduced by Buchweitz \cite{Buc21}.
The name ``singularity category" is justified by the fact that 
an algebra $A$ has finite global dimension if and only if $D_{sg}(A)$ vanishes.
The
singularity category captures the stable homological properties of an algebra.
As stated in the Buchweitz-Happel theorem that if $A$ is Iwanaga-Gorenstein,
the singularity category $D_{sg}(A)$ is triangulated equivalent to the stable category of maximal Cohen-Macaulay modules over $A$.
However, 
for non Iwanaga-Gorenstein algebras,
the singularity category is more difficult to describe.

In \cite{CY14}, Chen-Ye established a singular equivalence between an algebra and its idempotent subalgebra by contraction along a localisable object. 
By iteratively applying the result, they showed that the singularity category of a Nakayama algebra is triangulated equivalent to the stable category of a self-injective one. 
Shen \cite{Sh15} studied the singularity category of a Nakayama algebra in more details and gave an alternative approach for the same result.
Moreover, an explicit construction of such a singular equivalence between a Nakayama algebra and its idempotent subalgebra is given, which relies on the resolution quiver of a Nakayama algebra introduced by Ringel \cite{Rin13}.

The same question arises for higher Nakayama algebras. McMahon \cite{McM20} generalized Chen-Ye's approach to higher Auslander-Reiten theory.
By contracting along fabric idempotent ideals iteratively, he showed that the singularity category of a $d$-Nakayama algebra is triangulated equivalent to the stable module category of a self-injective $d$-Nakayama algebra which is an idempotent subalgebra of the original one.  

Motivated by the above results,
we generalize Shen's approach for higher Nakayama algebras and give an explicit construction of a singular equivalence between a $d$-Nakayama algebra $A$ and a self-injective $d$-Nakayama algebra $B$ which is an idempotent subalgebra of $A$. 
Moreover, this singular equivalence restricts to an equivalence between $(d+2)$-angulated categories $\underline{\add N} \xrightarrow[]{\sim} \underline{\add M}$, where $M$ (resp. $N$) is the distinguished $d\mathbb{Z}$-cluster tilting module of $A$ (resp. $B$).  
Our main result is stated  as follows.

\begin{Thm}\label{Thm}\cite[Theorem 4.4]{McM20} (Theorem 3.18)
    Let $A$ be a $d$-Nakayama algebra.
    Then the singularity category $D_{sg}(A)$ is triangulated equivalent to the stable module category $\underline{\modn } B$ with $B$ a self-injective $d$-Nakayama algebra. 
\end{Thm}

To get $B$,
we generalize the notion of resolution quiver to higher Nakayama algebras.
In precise terms,
if $A$ is a $d$-Nakayama algebra with Kupisch series $\underline{\ell} = (\ell_1, \ldots, \ell_n)$,
the vertex set of the resolution quiver $R(A)$ is $\{1, 2,\ldots, n\}$,
and there is an arrow from $i$ to $j$ if $j \equiv i - \ell_{i} -d + 1 \mod  n$.
When $d = 1$, our definition coincides with the resolution quivers for usual Nakayama algebras defined in \cite{Rin13}.
Let $J$ be the subset of $\{1, 2, \ldots, n\}$ which consists all the numbers that lie in a cycle of $R(A)$ and let $I = J + n\mathbb{Z}$.
It turns out that $B$ is the endomorphism algebra of the direct sum of indecomposable projective $A$-modules whose coordinates are in $I$.
Here by coordinates, we follow a slightly modified version of the notation given in \cite{JK19} where indecomposable $A$-modules are indexed by ordered sequences of length $d+1$ with certain restrictions given by $\underline{\ell}$.  

Additionally, the additive closure of modules in $\add M$ whose coordinates are in $I$ yields a $d$-wide subcategory $\mathcal{W}$ of $\add M$ in the sense of \cite{HJV20}.
By applying \cite[Theorem B]{HJV20}, we get an equivalence between $d$-abelian categories $\add N\xrightarrow{\sim} \mathcal{W}$.
Combined with Kvamme's result, the singular equivalence between $B$ and $A$ restricts to $\underline{\add N}\xrightarrow{\sim} \underline{\mathcal{W}} = \underline{\add M}$.
This allows us to calculate $D_{sg}(A)(M, M)$ via $\underline{\End}_B(N)$, which turns out to be a self-injective $(d+1)$-Nakayama algebra.
We thus show that this self-injective $(d+1)$-Nakayama algebra is twisted $(d+2)$-periodic.
Further, our result provides concrete examples of the derived Auslander-Iyama correspondence introduced by Jasso-Muro \cite{JKM22},
which states that there is a bijective correspondence between the equivalence classes of pairs $(\mathcal{T}, c)$ consisting of an algebraic Krull-Schmidt triangulated category $\mathcal{T}$ with a basic $d\mathbb{Z}$-cluster tilting object $c\in \mathcal{T}$ and the Morita equivalence classes of twisted $(d+2)$-periodic self-injective algebras by sending $c$ to $\mathcal{T}(c,c)$. 
In particular, this correspondence provides a method of recognizing such a pair $(\mathcal{T}, c)$ via $\mathcal{T}(c, c)$.
Several recognition theorems for algebraic triangulated categories were discussed in \cite[Section 6]{JKM22}.
In particular, \cite[Theorem 6.5.2]{JKM22} gave a recognition theorem of the stable module categories of self-injective higher Nakayama algebras.
Theorem \ref{Thm} extends this recognition theorem to the singularity categories of all higher Nakayama algebras.
Thus, we extend the library of algebraic triangulated categories that can be recognized via the derived Auslander-Iyama correspondence.

\textbf{Notation and conventions.}
Throughout this paper,
we fix positive integers $d$ and $n$.
We work over an arbitrary field $k$.
Unless stated otherwise,
all algebras are finite dimensional $k$-algebras and all modules are finite dimensional right modules.
We denote by $D$ the $k$-duality $\Hom_k(-, k)$.

All subcategories considered are supposed to be full.
Let $F: \mathcal{C}\rightarrow \mathcal{D}$ be a functor,
the essential image of $F$ is the full subcategory of $\mathcal{D}$ given by 
\begin{equation*}
    F\mathcal{C} = \{D\in \mathcal{D} \mid \exists  C\in \mathcal{C} \text{ such that } FC \cong D\}.
\end{equation*}
Let $\mathcal{C}$ be a triangulated category.
We denote by $\Sigma$ the suspension functor of $\mathcal{C}$.
By $\tri(E)$ we mean the smallest triangulated subcategory of $\mathcal{C}$ containing the set of objects $E$ in $\mathcal{C}$.

Let $A$ be a finite dimensional algebra over $k$ and $\modn A$ the category of finitely generated right $A$-modules.
We denote by $\underline{\modn}A$ the projectively stable module category of $A$,
that is the category with the same  objects as $\modn A$ and morphisms given by $\underline{\Hom}_A(M,N) = \Hom_A(M,N)/\mathcal{P}(M,N)$ where $\mathcal{P}(M, N)$ denotes the subspace of morphisms factoring through projective modules.
We denote by $\Omega: \underline{\modn}A\rightarrow \underline{\modn}A$ the syzygy functor defined by $\Omega(M)$ being the kernel of the projective cover $P(M)\twoheadrightarrow M$.
Let $\Omega^0(M) = M$ and $\Omega^{i+1}(M) = \Omega(\Omega^i(M))$ for $i\geq 0$.
The injectively stable module category $\overline{\modn}A$ of $A$ and the cosyzygy functor $\Omega^{-1}: \overline{\modn}A\rightarrow \overline{\modn}A$ are defined dually.
When $A$ is a self-injective algebra, $\modn A$ is Frobenius thus $\underline{\modn} A$ has a triangulated category structure with the suspension functor $\Omega^{-1}$.
We refer to \cite{Ha88} for more details.

We consider the $d$-Auslander-Reiten translations $\tau_d: \underline{\modn}A\rightarrow \overline{\modn}A$    and $\tau_d^{-}: \overline{\modn}A\rightarrow \underline{\modn}A$
defined by $\tau_d = \tau\Omega^{d-1}$ and $\tau_d^{-} = \tau^{-}\Omega^{-(d-1)}$
where $\tau$ and $\tau^{-}$ denote the usual Auslander-Reiten translations.

Recall that $D^b(\modn A)$ denotes the bounded derived category of $\modn A$.
The category $\modn A$ is a full subcategory of $D^b(\modn A)$ by identifying an $A$-module with the corresponding stalk complex concentrated at degree zero.
A complex in $D^b(\modn A)$ is perfect provided that it is quasi-isomorphic to a bounded complex of finitely generated projective $A$-modules.
Perfect complexes form a thick subcategory of $D^b(\modn A)$,
which is denoted by $\perf(A)$. 
The singularity category of $A$,
denoted by $D_{sg}(A)$ is the quotient of triangulated categories given as
\begin{equation*}
    D_{sg}(A) = D^b(\modn A)/\perf(A).
\end{equation*}
Recall that $K^{-,b}(\proj A)$ denotes the upper bounded homotopy category of $\proj A$ and $K^{b}(\proj A)$ denotes bounded homotopy category of $\proj A$, which is a thick triangulated subcategory of $K^{-,b}(\proj A)$.
Via the equivalences $K^{-,b}(\proj A) \cong D^b(\modn A)$ and $K^{b}(\proj A) \cong \perf(A)$, 
we have that 
\begin{equation*}
    D_{sg}(A)\cong K^{-,b}(\proj A) / K^{b}(\proj A).
\end{equation*}
Denote by $q': D^b(\modn A)\rightarrow D_{sg}(A)$ the quotient functor. 
Observe that the functor $\modn A\rightarrow D^b(\modn A) \xrightarrow[]{q'} D_{sg}(A)$ vanishes on projective modules.
Hence it induces a functor $q: \underline{\modn} A\rightarrow D_{sg}(A)$.

By a singular equivalence between two algebras $A$ and $B$, 
we mean a triangle equivalence between their singularity categories.

\section{Preliminaries}\label{sect2}

\subsection{$d$-cluster tilting subcategories}

Let $\mathcal{M}$ be a subcategory of a category $\mathcal{C}$ and let $C\in \mathcal{C}$.
A right $\mathcal{M}$-approximation of $C$ is a morphism $f: M\rightarrow C$ with $M\in \mathcal{M}$ such that all morphisms $g: M'\rightarrow C$ with $M'\in \mathcal{M}$ factor through $f$.
$\mathcal{M}$ is contravariantly finite in $\mathcal{C}$ if every $C\in \mathcal{C}$ admits a right $\mathcal{M}$-approximation.
The notions of left $\mathcal{M}$-approximation and covariantly finite are defined dually.
We say that $\mathcal{M}$ is functorially finite in $\mathcal{C}$ if $\mathcal{M}$ is both contravariantly finite and covariantly finite.
In particular, if $M\in \modn A$, 
then $\add M$ is functorially finite.
By a right (minimal) $\add M$-resolution of $X\in \modn A$ we mean the following complex
\begin{equation*}
    \xymatrix@R=1em@C=1em{\cdots\ar[r] & M_n\ar[r]^{f_n} & \cdots\ar[r] & M_1\ar[r]^{f_1} & M_0\ar[r]^{f_0} & X}
\end{equation*}
with $f_0$ a (minimal) right $\add M$-approximation of $X$ and $f_i$ a (minimal) right $\add M$-approximation of $\ker f_{i-1}$ for all $i\geq 1$.
A left (minimal) $\add M$-resolution is defined dually.
Recall in the case when $\mathcal{C}$ is abelian, $\mathcal{M}$ is called a generating (resp. cogenerating) subcategory if for any object $C\in\mathcal{C}$, there exists an epimorphism $M\rightarrow C$ (resp. monomorphism $C \rightarrow M$) with $M\in \mathcal{M}$.

\begin{Def}[\cite{Iya11, IY08, IJ17}]
    Let $d$ be a positive integer.
    Let $\mathcal{C}$ be an abelian or a triangulated category,
    and $A$ a finite-dimensional $k$-algebra.
    \begin{itemize}
        \item [(a)] We call a subcategory $\mathcal{M}$ of $\mathcal{C}$ a $d$-cluster tilting subcategory if it is functorially finite, generating-cogenerating if $\mathcal{C}$ is abelian and 
        \begin{align*}
            \mathcal{M} & = \{ C\in \mathcal{C} \mid \Ext_{\mathcal{C}}^i(C, \mathcal{M}) = 0 \text{ for } 1\leq i\leq d-1\}\\
            & = \{ C\in \mathcal{C} \mid \Ext_{\mathcal{C}}^i(\mathcal{M}, C) = 0 \text{ for } 1\leq i\leq d-1\}.
        \end{align*}
        If moreover $\Ext_{\mathcal{C}}^i(\mathcal{M}, \mathcal{M}) \neq 0$ implies that $i\in d\mathbb{Z}$,
        then we call $\mathcal{M}$ a $d\mathbb{Z}$-cluster tilting subcategory.
        \item [(b)] A finitely generated module $M\in \modn A$ is called a $d$-cluster tilting module (respectively $d\mathbb{Z}$-cluster tilting module) if $\add M$ is a $d$-cluster tilting subcategory (respectively $d\mathbb{Z}$-cluster tilting subcategory) of $\modn A$.
    \end{itemize}
\end{Def}

\begin{Rem}\label{dZ-CTInTriCat}
     If $\mathcal{C}$ is a triangulated category, then
    \begin{equation*}
        \Ext_{\mathcal{C}}^i(X, Y) = \Hom_{\mathcal{C}}(X,\Sigma^i Y) \text{ for } X, Y\in \mathcal{C}.
    \end{equation*}
    Therefore $\mathcal{M}$ is a $d\mathbb{Z}$-cluster tilting subcategory of $\mathcal{C}$ if and only if $\mathcal{M}$ is $d$-cluster tilting and $\Sigma^d \mathcal{M} \subset \mathcal{M}$.
    Note that in this case, $\mathcal{M}$ has a $(d+2)$-angulated structure in the sense of \cite{GKO13}. 
\end{Rem}

\begin{Def}\cite[Definition 2.5]{HJV20}
    We call $(A, \mathcal{M})$ a $d$-homological pair if $A$ is a finite dimensional $k$-algebra and $\mathcal{M} \subset \modn A$ is a $d$-cluster tilting subcategory.
\end{Def}

\begin{Prop}[\cite{Iya11}]\label{CTApprox}
    Let $(A, \mathcal{M})$ be a $d$-homological pair.
    Then each $X\in \modn A$ has a minimal right $\mathcal{M}$-resolution
    $\xymatrix@C=1em@R=1em{0\ar[r] & M_d\ar[r] & \cdots\ar[r] & M_1\ar[r] & X\ar[r] & 0}$ and a minimal left $\mathcal{M}$-resolution
    $\xymatrix@C=1em@R=1em{0\ar[r] & X\ar[r] & M_1'\ar[r] & \cdots\ar[r] & M_d'\ar[r] & 0}$,
    which are exact.
\end{Prop}

The following result relates $d\mathbb{Z}$-cluster tilting subcategories of $\modn A$ and $D_{sg}(A)$.

\begin{Thm}[\cite{Kv21}]\label{Thm_Kv}
    Let $A$ be a finite dimensional algebra and $\mathcal{M}$ a $d\mathbb{Z}$-cluster tilting subcategory of $\modn A$.
    Then the subcategory
    \begin{equation*}
        \underline{\mathcal{M}} = \{X\in D_{sg}(A) \mid X\cong M[di] \text{ for some } M\in \mathcal{M} \text{ and } i\in \mathbb{Z}\}
    \end{equation*}
    is an $d\mathbb{Z}$-cluster tilting subcategory of $D_{sg}(A)$. 
    In particular, $\underline{\mathcal{M}}$ is a $(d+2)$-angulated category.
\end{Thm}

\subsection{$d$-abelian categories and wide subcategories}

The notion of $d$-abelian categories was introduced by Jasso \cite{Jas16}. 
It relies on the notion of $d$-kernel, $d$-cokernel and $d$-extension,
which we now recall following \cite{HJ21}.

\begin{Def}\cite[Definition 2.1]{HJ21}
    Let $\mathcal{M}$ be an additive category and 
    \begin{equation*}
        \mathbb{E}: \xymatrix@C=1em@R=1em{M_{d+1}\ar[r]^f & M_d\ar[r] & \cdots\ar[r] & M_1\ar[r]^g & M_0}
    \end{equation*}
    a sequence in $\mathcal{M}$.
    \begin{itemize}
        \item [(1)] We call 
        \begin{equation*}
            \xymatrix@C=1em@R=1em{M_{d+1}\ar[r]^f & M_d\ar[r] & \cdots\ar[r] & M_1}
        \end{equation*}
        a $d$-kernel of $g$ if
        \begin{equation*}
             \xymatrix@C=1em@R=1em{0\ar[r] & \mathcal{M}(M, M_{d+1})\ar[r]^{f\circ -} & \mathcal{M}(M, M_d)\ar[r] & \cdots\ar[r] & \mathcal{M}(M, M_0)}
        \end{equation*}
        is exact for all $M\in \mathcal{M}$.
        \item [(2)] We call 
        \begin{equation*}
            \xymatrix@C=1em@R=1em{ M_d\ar[r] & \cdots\ar[r] & M_1\ar[r]^g & M_0}
        \end{equation*}
        a $d$-cokernel of $f$ if
        \begin{equation*}
            \xymatrix@C=1em@R=1em{0\ar[r] & \mathcal{M}(M_0, M)\ar[r]^{-\circ g} & \mathcal{M}(M_1, M)\ar[r] & \cdots\ar[r] & \mathcal{M}(M_{d+1}, M)}
        \end{equation*}
         is exact for all $M\in \mathcal{M}$.
         \item [(3)] If both (1) and (2) are satisfied we call $\mathbb{E}$ a $d$-exact sequence (or a $d$-extension of $M_0$ by $M_{d+1}$).
         \item [(4)] We say that $\mathcal{M}$ is $d$-abelian if it is idempotent split, every morphism admits a $d$-kernel and a $d$-cokernel,
         and every monomorphism $f$ respectively epimorphism $g$ fits into a $d$-exact sequence of the form $\mathbb{E}$.
    \end{itemize}
\end{Def}

As shown in \cite{Jas16}, $d$-cluster tilting subcategories of abelian categories are $d$-abelian. 
Now we recall the notion of wide subcategories of $d$-abelian categories.

\begin{Def}[\cite{HJV20}]
    An additive subcategory $\mathcal{W}$ of a $d$-abelian category $\mathcal{M}$ is called wide if it satisfies the following conditions:
    \begin{itemize}
        \item [(i)] Each morphism in $\mathcal{W}$ has a $d$-kernel in $\mathcal{M}$ which consists of objects from $\mathcal{W}$.
        \item [(ii)] Each morphism in $\mathcal{W}$ has a $d$-cokernel in $\mathcal{M}$ which consists of objects from $\mathcal{W}$.
        \item [(iii)] Each $d$-exact sequence in $\mathcal{M}$,
        \begin{equation*}
            \xymatrix@R=1em@C=1em{0\ar[r] & W'\ar[r] & M_d\ar[r] & \cdots\ar[r] & M_1\ar[r] & W''\ar[r] & 0},
        \end{equation*}
        with $W', W''\in \mathcal{W}$, is Yoneda equivalent to a $d$-exact sequence in $\mathcal{M}$,
        \begin{equation*}
         \xymatrix@R=1em@C=1em{0\ar[r] & W'\ar[r] & W_d\ar[r] & \cdots\ar[r] & W_1\ar[r] & W''\ar[r] &  0},
        \end{equation*}
        with $W_i\in \mathcal{W}$ for each $i$.
    \end{itemize}
\end{Def}

In the following theorem, 
statement $(a)$ provides a construction of wide subcategories inside $d$-cluster tilting subcategories which was obtained in \cite{HJV20}.
Here we have relaxed the condition $(iii)$.
The same proof and same statement in \cite[Theorem B]{HJV20}  still apply.
Moreover, based on the same setting, 
we obtain statement $(b)$,
which provides a singular equivalence between two $k$-algebras under suitable conditions.
If in addition, we require $\mathcal{M}$ in the $d$-homological pair $(A, \mathcal{M})$ to be $d\mathbb{Z}$-cluster tilting,
then by Theorem \ref{Thm_Kv},
we obtain statement $(c)$,
which gives an equivalence between two $(d+2)$-angulated categories in the sense of \cite{GKO13}.

\begin{Thm}\label{Thm_WideSubcat}
    Let $(A, \mathcal{M})$ be a $d$-homological pair. Let $\mathcal{W} \subset \mathcal{M}$ be an additive subcategory. Let $P\in \mathcal{W}$ be a module and set $B = \End_A(P)$, so that $P$ becomes a $B$-$A$-bimodule. Assume the following:
    \begin{itemize}
        \item [(i)] As an $A$-module $P$ has finite projective dimension.
        \item [(ii)] $\Ext_{A}^{ i}(P,P) = 0$ for all $i\geq 1$.
        \item [(iii)] Each $W\in \mathcal{W}$ admits an exact $\add P$-resolution
        \begin{equation*}
            \xymatrix@C=1em@R=1em{
                \cdots\ar[r] & P_m\ar[r] & \cdots\ar[r] & P_1\ar[r] & P_0\ar[r] & W\ar[r] & 0,
                }  P_i\in \add P.
        \end{equation*}
        \item [(iv)] $(B, \mathcal{N})$ is a $d$-homological pair where
        $i_{\rho} = \Hom_A(P, -): \modn A\rightarrow \modn B$ and $\mathcal{N} = i_{\rho}(\mathcal{W})$. 
    \end{itemize}
    Then the following statements hold
    \begin{itemize}
        \item [(a)] 
        $\mathcal{W}$ is a wide subcategory of $\mathcal{M}$ and there is an equivalence of categories 
    \begin{equation*}
        i_{\lambda} = -\otimes_{B} P: \mathcal{N} \rightarrow \mathcal{W}. 
    \end{equation*}
        \item [(b)] 
        $i_{\lambda}: \modn B\rightarrow \modn A$ induces a
        fully faithful triangle functor between the singularity categories of $A$ and $B$,
    \begin{equation*}
        D_{sg}(i_{\lambda}): D_{sg}(B) \rightarrow D_{sg}(A).
    \end{equation*}
    Moreover, $D_{sg}(i_{\lambda})$ is a triangle equivalence if 
    for any indecomposable $M\in \mathcal{M}$,
    there exists an integer $n\in \mathbb{N}$ such that $\Omega^n(M)\in \mathcal{W}$.
        \item [(c)] If in addition $\mathcal{M}$ is a $d\mathbb{Z}$-cluster tilting subcategory of $\modn A$,
        then 
        $\underline{\mathcal{M}} \subset D_{sg}(A)$, $\underline{\mathcal{N}}\subset D_{sg}(B)$ are $d\mathbb{Z}$-cluster tilting and hence $(d+2)$-angulated.
        Moreover,
        $D_{sg}(i_{\lambda})$ restricts to an equivalence between $(d+2)$-angulated categories
        \begin{equation*}
            D_{sg}(i_{\lambda}): \underline{\mathcal{N}} \rightarrow \underline{\mathcal{M}}.
        \end{equation*}
    \end{itemize}
\end{Thm}
\begin{proof}
    We refer to \cite[Section 3]{HJV20} for the proof of $(a)$.
    We remind the reader that condition $(iii)$ in \cite[Theorem B]{HJV20} requires such a resolution to be finite,
    while we allow it here to be infinite.
    The same proof applies.
    
    Now we prove $(b)$.
    As in the proof of \cite[Theorem B]{HJV20}, 
    we may apply \cite[Lemma 3.3]{HJV20} to get that $_BP$ is projective.
    Hence $i_{\lambda}$ is an exact functor.
    Therefore $i_{\lambda}$ extends to a triangle functor
    \begin{equation*}
         i_{\lambda}^{\ast}: D^b(\modn B) \rightarrow D^b(\modn A).
    \end{equation*}
    By \cite[Lemma 3.6]{HJV20}, $i_{\lambda}^{\ast}$ is fully faithful.
    
    Note that $i_{\rho}$ and $i_{\lambda}$ restrict to quasi-inverse equivalences between $\add P$ and $\add B$.
    Hence $i_{\lambda}^{\ast}$ preserves perfect complexes since $P\in \perf(A)$ by (i).
    Therefore $i_{\lambda}^{\ast}$ induces a triangle functor between singularity categories
    \begin{equation*}
        D_{sg}(i_{\lambda}): D_{sg}(B) \rightarrow D_{sg}(A).
    \end{equation*}
    
    Next we show that $D_{sg}(i_{\lambda})$ is fully faithful.
    Firstly we claim that $$i_{\lambda}^{\ast}(\perf(B)) = i_{\lambda}^{\ast}(D^b(\modn B))\cap \perf(A).$$
    Again $P\in\perf(A)$ shows 
    $i_{\lambda}^{\ast}(\perf(B)) \subset i_{\lambda}^{\ast}(D^b(\modn B))\cap \perf(A)$.
    To see the other inclusion, 
    we take $X_{\bullet}\in D^b(\modn B)$ such that $i_{\lambda}^{\ast}(X_{\bullet})\in \perf(A)$.
    Choose $Q_{\bullet}\in K^{-,b}(\proj B)$ such that $Q_{\bullet}\cong X_{\bullet}$ in $D^b(\modn B)$.
    Let $n$ be the largest integer such that $H_n(Q_{\bullet}) \neq 0$.
    Denote by $\sigma_{\geq n}Q_{\bullet}$ the brutal truncation of $Q_{\bullet}$ at degree $\geq n$.
    Then $\sigma_{\geq n}Q_{\bullet}\cong \Sigma^n(M)$ in $D^b(\modn B)$ for some $M\in \modn B$.
    We have the following triangle
    \begin{equation*}
        \sigma_{< n}Q_{\bullet}\rightarrow Q_{\bullet} \rightarrow \sigma_{\geq  n}Q_{\bullet}\rightarrow\Sigma\sigma_{< n}Q_{\bullet}.
    \end{equation*}
    Applying $i_{\lambda}^{\ast}$ to it yields a triangle in $D^b(\modn A)$
    \begin{equation*}
        i_{\lambda}^{\ast}(\sigma_{< n}Q_{\bullet})\rightarrow i_{\lambda}^{\ast}(Q_{\bullet}) \rightarrow i_{\lambda}^{\ast}(\sigma_{\geq  n}Q_{\bullet})\rightarrow i_{\lambda}^{\ast}(\Sigma\sigma_{< n}Q_{\bullet}).
    \end{equation*}
    Note that $\sigma_{< n}Q_{\bullet}\in \perf(B)$, 
    so $i_{\lambda}^{\ast}(\sigma_{< n}Q_{\bullet})\in \perf(A)$.
    By assumption $i_{\lambda}^{\ast}(Q_{\bullet}) \cong i_{\lambda}^{\ast}(X_{\bullet})\in \perf(A)$.
    Hence $i_{\lambda}^{\ast}(\sigma_{\geq n}Q_{\bullet})\in \perf(A)$ which yields $\Sigma^{-n} i_{\lambda}^{\ast}(\sigma_{\geq n}Q_{\bullet})\cong i_{\lambda}^{\ast}(M)\in \perf(A)$.
    If $M$ were not in $\perf(B)$, then $\projdim M_B = \infty$.
    That is, for any positive integer $i$, $\Ext_B^i(M, \Omega^iM) \neq 0$.
    Then $\Ext_A^i(i_{\lambda}(M), i_{\lambda}(\Omega^iM)) \cong \Ext_B^i(M, \Omega^iM) \neq 0$ by \cite[Lemma 3.6]{HJV20}.
    But this contradicts with the fact that $i_{\lambda}^{\ast}(M)\in \perf(A)$.
    So $M\in \perf(B)$ and $i_{\lambda}^{\ast}(\sigma_{\geq n}Q_{\bullet})\in i_{\lambda}^{\ast}(\perf(B))$.
    Using the above triangle we get $i_{\lambda}^{\ast}(X_{\bullet})\cong i_{\lambda}^{\ast}(Q_{\bullet})\in i_{\lambda}^{\ast}(\perf(B))$.
    
    So we have the following commutative diagram.
    \begin{equation*}
        \xymatrix{
        D^b(\modn B)\ar[r]^{\sim} & i_{\lambda}^{\ast}(D^b(\modn B))\ar@{^{(}->}[r] & D^b(\modn A)\\
        \perf(B)\ar[r]^-{\sim}\ar@{^{(}->}[u] & i_{\lambda}^{\ast}(\perf(B)) = \perf(A)\cap i_{\lambda}^{\ast}(D^b(\modn B))\ar@{^{(}->}[r]\ar@{^{(}->}[u] & \perf(A).\ar@{^{(}->}[u]
        }
    \end{equation*}

    Thus it suffices to show that the induced functor
    \begin{equation*}
        i_{\lambda}^{\ast}(D^b(\modn B))/(\perf(A)\cap i_{\lambda}^{\ast}(D^b(\modn B)))\rightarrow D_{sg}(A)
    \end{equation*}
    is fully faithful.
    To do this we apply \cite[Proposition 10.2.6]{KS06}.

    Let $f: X_{\bullet}\rightarrow Y_{\bullet}$ be a morphism in $D^b(\modn A)$ where $X_{\bullet} \in \perf(A)$ and $Y_{\bullet}\in i_{\lambda}^{\ast}(D^b(\modn B))$.
    We may assume that $X_{\bullet}\in K^b(\proj A)$ and $Y_{\bullet}$ is an upper bounded complex with terms in $\add P$.
    Suppose $X_i = 0$ for some $i > m$ and let $Z_{\bullet} = \sigma_{\leq m} Y_{\bullet} $ be the brutal truncation of $Y_{\bullet}$ at degree $\leq m$. 
    Since
    \begin{equation*}
        \Hom_{D^b(\modn A)}(X_{\bullet}, Y_{\bullet}) \cong \Hom_{K^-(\modn A)}(X_{\bullet}, Y_{\bullet}),
    \end{equation*}
    $f$ factors through $Z_{\bullet}$.
    Note that $Z_{\bullet}$ is a bounded complex with terms in $\add P$, so $Z_{\bullet}\in\perf(A)\cap i_{\lambda}^{\ast}(D^b(\modn B))$ as $P\in\perf(A)$.
    Therefore \cite[Proposition 10.2.6]{KS06} applies and the functor $D_{sg}(i_{\lambda})$ is fully faithful.

    It remains to show that $D_{sg}(i_{\lambda})$ is dense.
    Since by \cite[Corollary 2.3]{Che18}, 
    the singularity category of $A$ is the stabilization of $\underline{\modn} A$,
    the essential image $\mathcal{F}$ of 
    $D_{sg}(i_{\lambda})$ contains all objects $X\in \underline{\modn} A$ such that $\Omega^s(X) \cong i_{\lambda}(Y)$ for some $s\in \mathbb{N}$ and $Y\in \underline{\modn} B$.
    
    By $(a)$, the functor $i_{\lambda}$ induces an equivalence $i_{\rho}(\mathcal{W}) \xrightarrow{\sim} \mathcal{W}$.
    Thus by the condition given for $\mathcal{M}$ in $(b)$,
    we have that $\mathcal{M} \subset \mathcal{F}$.

    By Proposition \ref{CTApprox}, each $X\in \modn A$ admits a right minimal $\mathcal{M}$-resolution 
    \begin{equation*}
        \xymatrix@C=1em@R=1em{
        0\ar[r] & M_d\ar[r] & \cdots \ar[r] & M_2\ar[r] & M_1\ar[r] & X\ar[r] & 0
        }
    \end{equation*}
    which is exact and so we have that $X\cong M_{\bullet}$ in $D^b(\modn A)$ where 
    \begin{equation*}
        \xymatrix@C=1em@R=1em{
        M_{\bullet}: 0\ar[r] & M_d\ar[r] & \cdots \ar[r] & M_2\ar[r] & M_1\ar[r] & 0
        }.
    \end{equation*}
    Hence $X\in \tri(\mathcal{M})$ which implies that $D^b(\modn A) = \tri(\mathcal{M})$.
    
    Since $\mathcal{F}$ is a triangulated subcategory of $D_{sg}(A)$,
    it follows that $\mathcal{F} = D_{sg}(A)$.
    In other words, $D_{sg}(i_{\lambda})$ is dense.
    
    To see (c), we have that $\underline{\mathcal{M}}$ is a $d\mathbb{Z}$-cluster tilting subcategory of $D_{sg}(A)$ and hence a $(d+2)$-angulated category by Theorem \ref{Thm_Kv}.
    Then $\Sigma^{d} \underline{\mathcal{M}} \subset \underline{\mathcal{M}}$ by Remark \ref{dZ-CTInTriCat}.
    Since $i_{\lambda}: \mathcal{N} \xrightarrow{\sim} \mathcal{W}$ by (a), 
    we have that $D_{sg}(i_{\lambda})( \underline{\mathcal{N}}) = \underline{\mathcal{W}}$.
    Moreover $ \underline{\mathcal{W}} = \underline{\mathcal{M}}$ by (b).
    Hence $\Sigma'^d(\underline{\mathcal{N}}) \subset \underline{\mathcal{N}}$ which yields that $\underline{\mathcal{N}}$ is a $d\mathbb{Z}$-cluster tilting subcategory of $D_{sg}(B)$. 
    Here $\Sigma'$ is the suspension functor of $D_{sg}(B)$.
    Therefore, by restricting $D_{sg}(i_{\lambda})$ to $\underline{\mathcal{N}}$, 
    we have an equivalence $\underline{\mathcal{N}} \cong \underline{\mathcal{M}}$ as $(d+2)$-angulated categories.
\end{proof}

\begin{Rem}
    If $P$ is a projective $A$-module, condition $(i)$ and $(ii)$ are automatically satisfied. 
    Condition $(iii)$ is fulfilled if each $W\in \mathcal{W}$ has a projective resolution with all terms in $\add P$.
\end{Rem}

\subsection{Higher Nakayama algebras}\label{HighNakAlg}
We recall some definitions and basic facts about higher Nakayama algebras constructed by Jasso-K\"{u}lshammer \cite{JK19}.
We follow the notations in their paper with a slight modification (see Remark \ref{RemCoordNota}).

Recall that $\ell_{\infty} = (\ldots, \ell_{-1}, \ell_0, \ell_1, \ldots)$ is called a Kupisch series of type $A_{\infty}^{\infty}$ if for all $i\in \mathbb{Z}$ there are inequalities $1\leq \ell_i \leq \ell_{i-1} + 1$. 
\begin{itemize}
    \item $\ell_{\infty}$ is connected if  $\ell_i \geq 2$ for all $ i\in \mathbb{Z}$.
    \item 
    $\ell_{\infty}$ is $\ell$-bounded for some positive integer $\ell$ if $\ell_i \leq \ell$ for all $ i\in \mathbb{Z}$.
    \item $\ell_{\infty}$ is $n$-periodic if $\ell_i = \ell_{i + n}$ for $n\in \mathbb{N}$, in this case $\underline{\ell}$ is called Kupisch series of type $\widetilde{\mathbb{A}}_{n-1}$ and we use the notation $\underline{\ell} = (\ell_1, \ell_2, \ldots, \ell_n)$. 
\end{itemize}
Denote by $\ell_{\infty}[1] = (\ldots, \ell_{-1}[1], \ell_0[1], \ldots)$ the Kupisch series obtained from $\ell_{\infty}$ by letting $\ell_i[1] = \ell_{i+1}$.

Let $d$ be a positive integer.
We recall the definition of ordered sequences $(os_{\ell_{\infty}}^d, \preccurlyeq)$ from \cite{JK19}
\begin{equation*}
        os_{\ell_{\infty}}^d := \{x = (x_1, x_2, \ldots, x_d)  \mid x_1 < x_2 < \cdots < x_d \text{ and } x_d - x_1 + 1 \leq \ell_{x_d} + d - 1\},
\end{equation*}
with the relation $\preccurlyeq$ defined as $x \preccurlyeq y$ if $x_1 \leq y_1 < x_2\leq y_2 < \cdots < x_d \leq y_d$ for $x=(x_1, \ldots, x_d), y = (y_1, \ldots, y_d) \in os_{\ell_{\infty}}^d$. 

Now we describe the $d$-Nakayama algebra of type $A_{\infty}^{\infty}$ with Kupisch series $\ell_{\infty}$ by quiver with relations.
The set of vertices of quiver $Q_{\ell_{\infty}}^d$ is the set $os_{\ell_{\infty}[1]}^d$.
Let $\{e_i\mid 1\leq i\leq n\}$ be the standard basis of $\mathbb{Z}^n$.
There is an arrow $a_i(x) : x\rightarrow x + e_i$ whenever $x + e_i \in os_{\ell_{\infty}[1]}^d$.
Let $I$ be the ideal of the path category $kQ_{\ell_{\infty}}^d$ generated by  $a_i(x + e_j)a_j(x) - a_j(x + e_i)a_i(x)$ with $1\leq i,j\leq n$. 
By convention, $a_i(x) = 0$ whenever $x$ or $x+e_i$ is not in  $os_{\ell_{\infty}[1]}^d$, hence some of the relations are indeed zero relations.
Then the $d$-Nakayama algebra of type $A_{\infty}^{\infty}$ with Kupisch series $\ell_{\infty}$ is given by $\mathcal{A}_{\ell_{\infty}}^{(d)} = kQ_{\ell_{\infty}}^d / I$.
\begin{Rem}\label{RemCoordNota}
    \begin{itemize}
        \item [(i)] The definition of $os_{\ell_{\infty}}^{d}$ is slightly different from that in \cite[Definition 1.9]{JK19}.
        We add $(1,2,\ldots,d)$ to each of the ordered sequence defined in \cite{JK19} to make it strictly increasing.
        \item [(ii)] By definition $\mathcal{A}_{\ell_{\infty}}^{(d)}$ is a locally bounded $k$-linear category.
        By abuse of notation,
        we still call it an algebra.
        We also identify categories with finitely many objects and algebras.
    \end{itemize}
\end{Rem}
By construction in \cite{JK19}, $\mathcal{A}_{\ell_{\infty}}^{(d)}$ has a distinguished $d\mathbb{Z}$-cluster tilting subcategory  $$\mathcal{M}_{\ell_{\infty}}^{(d)} = \{\widehat{M}(x) \mid x\in os_{\ell_{\infty}}^{d+1}\}.$$
Here as a representation $\widehat{M}(x)$ assigns $k$ to vertex $z\in os_{\ell_{\infty}[1]}^{d}$ if $(x_1, \ldots, x_d) \preccurlyeq z \preccurlyeq (x_2-1, \ldots, x_{d+1}-1)$ and $0$ otherwise.
All arrows $k\rightarrow k$ act as identity, while other arrows act as zero. 

Then
$$\Hom_{\mathcal{A}_{\ell_{\infty}}^{(d)}}(\widehat{M}(x),\widehat{M}(y)) \cong \left\{ \begin{array}{ll}
        kf_{yx} &  x \preccurlyeq y\\
        0 & \text{ otherwise.}
    \end{array} 
    \right.$$
Here $f_{yx}$ is given by $k\xrightarrow{1} k$ at vertices $z$ where $\widehat{M}(x)_z = \widehat{M}(y)_z = k$ and $0$ otherwise.
The composition of morphisms in $\mathcal{M}_{\ell_{\infty}}^{(d)}$ is completely determined by
\begin{equation*}
    f_{zy}\circ f_{yx} =
    \left\{ \begin{array}{ll}
        f_{zx} &  x \preccurlyeq z\\
        0 & \text{ otherwise.}
    \end{array} 
    \right.
\end{equation*}
In the following proposition, we recall some homological properties of modules in  $\mathcal{M}_{\ell_{\infty}}^{(d)}$ described by combinatorial data.
\begin{Prop}\label{Prop_modcomb}\cite[Proposition 2.22, Proposition 2.25, Theorem 3.16]{JK19}
Let $x\in os_{\ell_{\infty}}^{d+1}$. The following statements hold.
    \begin{itemize}
      \item [(i)] $\top\widehat{M}(x) = S_{(x_2 - 1, \ldots, x_{d+1} - 1)}$ and $\soc \widehat{M}(x) = S_{(x_1, \ldots, x_d)}$. 
      \item [(ii)] $\widehat{M}(x)$ is simple if and only if $x = (i, i+1, \ldots, i + d)$ for some integer $i$.  
      \item [(iii)] $\widehat{M}(x)$ is projective if and only if $x_1 = \min \{y \mid (y, x_2, \ldots, x_{d+1}) \in os_{\ell_{\infty}}^{d+1}\}$, or equivalently,
      $x_1 = x_{d+1} - \ell_{x_{d+1}} - d  + 1$.
      \item [(iv)] 
      $\widehat{M}(x)$ is injective if and only if $x_{d+1} = \max\{y \mid (x_1, \ldots, x_d, y)\in os_{\ell_{\infty}}^{d+1}\}$.
      \item [(v)] If $x_1 > x_{d+1} -\ell_{x_{d+1}} - d + 1$, then there exists an exact sequence
          \begin{equation*}
              \xymatrix@=1em@R=1em{
                0\ar[r] & \Omega^d(\widehat{M}(x))\ar[r] & P_d\ar[r] & \cdots\ar[r] & P_1\ar[r] & \widehat{M}(x)\ar[r] & 0
              }
          \end{equation*}
        with $P_i = \widehat{M}(x_{d+1} - \ell_{x_{d+1}} - d + 1, x_1, \ldots, x_{i-1}, x_{i+1}, \ldots, x_{d+1})$ for $1\leq i\leq d$ and $\Omega^d(\widehat{M}(x)) = \widehat{M}(x_{d+1} - \ell_{x_{d+1}} - d + 1, x_1, \ldots, x_d)$. 
      \item [(vi)] $\tau_d(\widehat{M}(x)) = \widehat{M}(x_1 - 1, x_2 - 1, \ldots, x_{d+1} - 1)$.
      \item [(vii)] For each $i\in \mathbb{Z}$ the indecomposable projective $\mathcal{A}_{\ell_{\infty}}^{(d)}$-module at the vertex $(i-d + 1, \ldots, i)$ has Loewy length $\ell_i$.
  \end{itemize}
\end{Prop}

Recall that for a locally bounded $k$-linear category $\mathcal{C}$, 
a group action given by $G$ is called admissible if $gx\ncong x$ for any indecomposable object $x$ in $\mathcal{C}$ and $g\in G\backslash \{1\}$.

Let $n$ be a fixed positive integer.
From now on we assume $\ell_{\infty}$ is $n$-periodic and let $\underline{\ell} = (\ell_1, \ldots, \ell_n)$.
Then $G=\langle \sigma\rangle$ where $\sigma=\tau_d^n$, is an 
 admissible group acting on $\mathcal{A}_{\ell_{\infty}}^{(d)}$.
 Jasso-K\"{u}lshammer \cite{JK19} constructed $d$-Nakayama algebras of type $\widetilde{\mathbb{A}}_{n-1}$ as the orbit category
\begin{equation*}
    A_{\underline{\ell}}^{(d)} := \mathcal{A}_{\ell_{\infty}}^{(d)} / G.
\end{equation*}

The covering functor:
\begin{equation*}
    F: \mathcal{A}_{\ell_{\infty}}^{(d)} \rightarrow A_{\underline{\ell}}^{(d)}
\end{equation*}
induces an exact functor
\begin{equation*}
    F^{\ast}: \Mod A_{\underline{\ell}}^{(d)} \rightarrow \Mod \mathcal{A}_{\ell_{\infty}}^{(d)}
\end{equation*}
called pull-up, 
given by $F^{\ast}(M) = M\circ F$.
This functor has a left adjoint 
\begin{equation*}
    F_{\ast}: \Mod \mathcal{A}_{\ell_{\infty}}^{(d)} \rightarrow \Mod A_{\underline{\ell}}^{(d)}
\end{equation*}
called push-down, which is also exact.
In particular, $F_{\ast}$ induces a functor 
$F_{\ast}: \modn \mathcal{A}_{\ell_{\infty}}^{(d)} \rightarrow \modn A_{\underline{\ell}}^{(d)}$ between the category of finitely generated modules of $\mathcal{A}_{\ell_{\infty}}^{(d)}$ and $A_{\underline{\ell}}^{(d)}$ respectively.\\
The $d\mathbb{Z}$-cluster tilting subcategory  $\mathcal{M}_{\ell_{\infty}}^{(d)}$ is $G$-equivariant, i.e. $\mathcal{M}_{\ell_{\infty}}^{(d)}$ and $\sigma_{\ast}(\mathcal{M}_{\ell_{\infty}}^{(d)})$ have the same isomorphism closure in $\modn \mathcal{A}_{\ell_{\infty}}^{(d)}$ where $\sigma_{\ast}: \modn \mathcal{A}_{\ell_{\infty}}^{(d)}\rightarrow \modn\mathcal{A}_{\ell_{\infty}}^{(d)}$ is the induced automorphism of module category defined by precomposition with $\sigma^{-1}$. 
Then \cite[Theorem 2.3]{DI20} implies that $F_{\ast}\mathcal{M}_{\ell_{\infty}}^{(d)}$ is a $d\mathbb{Z}$-cluster tilting subcategory of $\modn A_{\underline{\ell}}^{(d)}$, which we denote by $\mathcal{M}_{\underline{\ell}}^{(d)}$, that is
\begin{equation*}\label{CT_Lift}
    \mathcal{M}_{\underline{\ell}}^{(d)} = F_{\ast}\mathcal{M}_{\ell_{\infty}}^{(d)} = \add\{M(x) \mid x\in os_{\ell_{\infty}}^{d+1}\} \text{ where } M(x) = F_{\ast}\widehat{M}(x) .
\end{equation*}
Note that $M(x) \cong M(\sigma(x))$ and
the orbit category $\mathcal{M}_{\underline{\ell}}^{(d)}$ is a graded category with the natural grading given by $G$.
More precisely,
take $M(x), M(y)\in \mathcal{M}_{\underline{\ell}}^{(d)}$ for $x, y\in os_{\ell_{\infty}}^{d+1}$. 
Then
\begin{equation*}
    \Hom_{A_{\underline{\ell}}^{(d)}}(M(x), M(y)) = \bigoplus_{i = a_{yx}}^{b_{yx}}\Hom_{\mathcal{A}_{\ell_{\infty}}^{(d)}}(\widehat{M}(x), \widehat{M}(\sigma^i(y))).
\end{equation*}
Here $a_{yx}$ (resp.$b_{yx}$) is the minimal (resp.maximal) integer such that $x\preccurlyeq \sigma^i(y)$.
Denote by $f_{yx}^i\in \Hom_{A_{\underline{\ell}}^{d}}(M(x),M(y))$ the image of $f_{\sigma^i(y),x}\in \Hom_{\mathcal{A}_{\ell_{\infty}}^{(d)}}(\widehat{M}(x) , \widehat{M}(\sigma^i(y)))$ in $\mathcal{M}_{\underline{\ell}}^{(d)}$.
By the composition law in $\mathcal{M}_{\ell_{\infty}}^{(d)}$, we have 
\begin{equation*}
    f_{zy}^j\circ f_{yx}^i = \left\{ \begin{array}{ll}
        f_{zx}^{i+j} &  x \preccurlyeq \sigma^{i+j}(z)\\
        0 & \text{ otherwise.}
    \end{array} 
    \right.
\end{equation*}
Note that $\dim_k\Hom_{A_{\underline{\ell}}^{(d)}}(M(x), M(y)) = b_{yx} - a_{yx} + 1$.

\begin{Rem}
     A non-semisimple $d$-Nakayama algebra of type  $\widetilde{\mathbb{A}}_{n-1}$ is self-injective if and only if $\underline{\ell} = (\ell, \ldots, \ell)$ for some integer $\ell \geq 2$ \cite[Theorem 4.10]{JK19}. 
In this case, we denote $A_{\underline{\ell}}^{(d)}$ by $A_{n, \ell}^{(d)}$ and its distinguished $d\mathbb{Z}$-cluster tilting subcategory by $\mathcal{M}_{n, \ell}^{(d)}$ following the notations in \cite[Section 4.1]{JK19}.
\end{Rem}
\begin{Eg}\label{egCTSubCat}
    Let $d = 2$, $n=5$ and $\underline{\ell} = (3, 4, 4, 4, 4)$ be a periodic Kupisch series.
    Then the Gabriel quiver $Q_{\underline{\ell}}^{(2)}$ of $A_{\underline{\ell}}^{(2)}$ is given as follows.
    
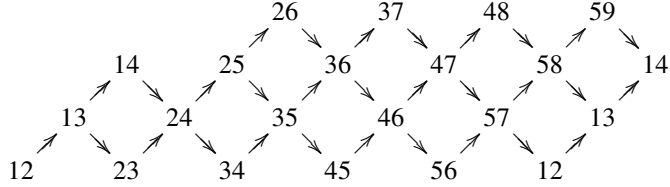
\begin{equation*}
        \begin{xy}
            0;<20pt,0cm>:<20pt,20pt>::
            (0,0) *+{12} ="12",
            (0,1) *+{13} ="13",
            (0,2) *+{14} ="14",
            (2,0) *+{23} ="23",
            (2,1) *+{24} ="24",
            (2,2) *+{25} ="25",
            (2,3) *+{26} ="26",
            (4,0) *+{34} ="34",
            (4,1) *+{35} ="35",
            (4,2) *+{36} ="36",
            (4,3) *+{37} ="37",
            (6,0) *+{45} ="45",
            (6,1) *+{46} ="46",
            (6,2) *+{47} ="47",
            (6,3) *+{48} ="48",
            (8,0) *+{56} ="56",
            (8,1) *+{57} ="57",
            (8,2) *+{58} ="58",
            (8,3) *+{59} ="59",
            (10,0) *+{12} ="67",
            (10,1) *+{13} ="68",
            (10,2) *+{14} ="69",
            "12", {\ar "13"},
            "13", {\ar "14"},
            "23", {\ar "24"},
            "24", {\ar "25"},
            "25", {\ar "26"},
            "34", {\ar "35"},
            "35", {\ar "36"},
            "36", {\ar "37"},
            "45", {\ar "46"},
            "46", {\ar "47"},
            "47", {\ar "48"},
            "56", {\ar "57"},
            "57", {\ar "58"},
            "58", {\ar "59"},
            "67", {\ar "68"},
            "68", {\ar "69"},
             "13", {\ar "23"},
             "14", {\ar "24"},
             "24", {\ar "34"},
             "25", {\ar "35"},
             "26", {\ar "36"},
             "35", {\ar "45"},
             "36", {\ar "46"},
             "46", {\ar "56"},
             "37", {\ar "47"},
             "47", {\ar "57"},
             "57", {\ar "67"},
             "48", {\ar "58"},
             "58", {\ar "68"},
             "59", {\ar "69"},
        \end{xy}
    \end{equation*}
    Here the leftmost and the rightmost lines should be identified.
    The modules $M(358)$ and $M(368)$ are given as representations of $Q_{\underline{\ell}}^{(2)}$.
    Note that both of them are projective $A_{\underline{\ell}}^{(2)}$-modules.
    
    $M(358)$:
    
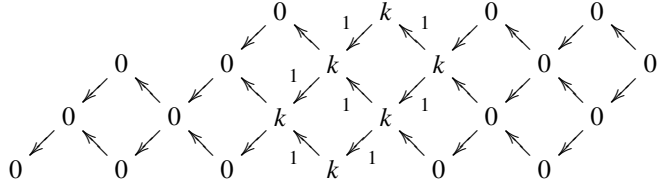
\begin{equation*}
        \begin{xy}
            0;<20pt,0cm>:<20pt,20pt>::
            (0,0) *+{0} ="12",
            (0,1) *+{0} ="13",
            (0,2) *+{0} ="14",
            (2,0) *+{0} ="23",
            (2,1) *+{0} ="24",
            (2,2) *+{0} ="25",
            (2,3) *+{0} ="26",
            (4,0) *+{0} ="34",
            (4,1) *+{k} ="35",
            (4,2) *+{k} ="36",
            (4,3) *+{k} ="37",
            (6,0) *+{k} ="45",
            (6,1) *+{k} ="46",
            (6,2) *+{k} ="47",
            (6,3) *+{0} ="48",
            (8,0) *+{0} ="56",
            (8,1) *+{0} ="57",
            (8,2) *+{0} ="58",
            (8,3) *+{0} ="59",
            (10,0) *+{0} ="67",
            (10,1) *+{0} ="68",
            (10,2) *+{0} ="69",
            "13", {\ar "12"},
            "14", {\ar "13"},
            "24", {\ar "23"},
            "25", {\ar "24"},
            "26", {\ar "25"},
            "35", {\ar "34"},
            "36", {\ar_1 "35"},
            "37", {\ar_1 "36"},
            "46", {\ar^1 "45"},
            "47", {\ar^1 "46"},
            "48", {\ar "47"},
            "57", {\ar "56"},
            "58", {\ar "57"},
            "59", {\ar "58"},
            "68", {\ar "67"},
            "69", {\ar "68"},
             "23", {\ar "13"},
             "24", {\ar "14"},
             "34", {\ar "24"},
             "35", {\ar "25"},
             "36", {\ar "26"},
             "45", {\ar^1 "35"},
             "46", {\ar^1 "36"},
             "56", {\ar "46"},
             "47", {\ar_1 "37"},
             "57", {\ar "47"},
             "67", {\ar "57"},
             "58", {\ar "48"},
             "68", {\ar "58"},
             "69", {\ar "59"},
        \end{xy}
    \end{equation*}
    
    $M(368)$:
    
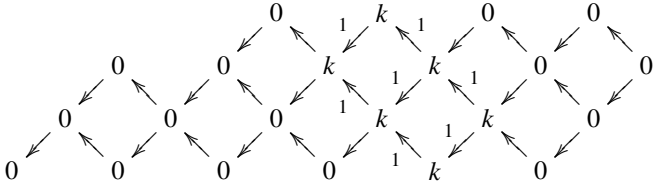
\begin{equation*}
        \begin{xy}
            0;<20pt,0cm>:<20pt,20pt>::
            (0,0) *+{0} ="12",
            (0,1) *+{0} ="13",
            (0,2) *+{0} ="14",
            (2,0) *+{0} ="23",
            (2,1) *+{0} ="24",
            (2,2) *+{0} ="25",
            (2,3) *+{0} ="26",
            (4,0) *+{0} ="34",
            (4,1) *+{0} ="35",
            (4,2) *+{k} ="36",
            (4,3) *+{k} ="37",
            (6,0) *+{0} ="45",
            (6,1) *+{k} ="46",
            (6,2) *+{k} ="47",
            (6,3) *+{0} ="48",
            (8,0) *+{k} ="56",
            (8,1) *+{k} ="57",
            (8,2) *+{0} ="58",
            (8,3) *+{0} ="59",
            (10,0) *+{0} ="67",
            (10,1) *+{0} ="68",
            (10,2) *+{0} ="69",
            "13", {\ar "12"},
            "14", {\ar "13"},
            "24", {\ar "23"},
            "25", {\ar "24"},
            "26", {\ar "25"},
            "35", {\ar "34"},
            "36", {\ar "35"},
            "37", {\ar_1 "36"},
            "46", {\ar "45"},
            "47", {\ar_1 "46"},
            "48", {\ar "47"},
            "57", {\ar_1 "56"},
            "58", {\ar "57"},
            "59", {\ar "58"},
            "68", {\ar "67"},
            "69", {\ar "68"},
             "23", {\ar "13"},
             "24", {\ar "14"},
             "34", {\ar "24"},
             "35", {\ar "25"},
             "36", {\ar "26"},
             "45", {\ar "35"},
             "46", {\ar^1 "36"},
             "56", {\ar^1 "46"},
             "47", {\ar_1 "37"},
             "57", {\ar_1 "47"},
             "67", {\ar "57"},
             "58", {\ar "48"},
             "68", {\ar "58"},
             "69", {\ar "59"},
        \end{xy}.
    \end{equation*}
    There is a nonzero morphism $\phi : M(358)\rightarrow M(368)$ since $(358) \preccurlyeq (368)$. 
    Indeed, $\phi_{36}=\phi_{37}=\phi_{46}=\phi_{47}=1_k$ and $\phi_{ij}=0$ for all other indices $ij$.
    
    The Auslander-Reiten quiver of the distinguished $2\mathbb{Z}$-cluster tilting subcategory $\mathcal{M}_{\underline{\ell}}^{(2)}$ is given below.
   \begin{equation*}
{\tiny
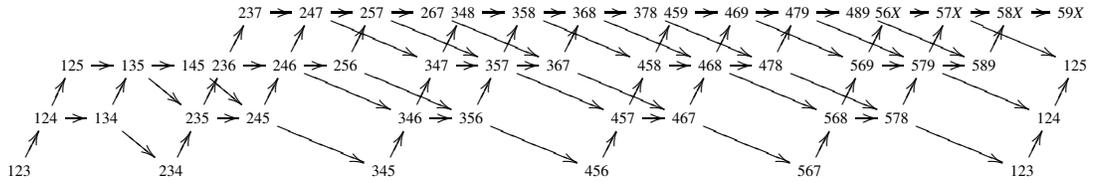
\begin{xy}
0;<23pt,0cm>:<10pt,20pt>:: 
(0,0) *+{123} ="123",
(0,1) *+{124} ="124",
(0,2) *+{125} ="125",
(1,1) *+{134} ="134",
(1,2) *+{135} ="135",
(2,2) *+{145} ="145",
(2.5,0) *+{234} ="234",
(2.5,1) *+{235} ="235",
(2.5,2) *+{236} ="236",
(2.5,3) *+{237} ="237",
(3.5,1) *+{245} ="245",
(3.5,2) *+{246} ="246",
(3.5,3) *+{247} ="247",
(4.5,2) *+{256} ="256",
(4.5,3) *+{257} ="257",
(5.5,3) *+{267} ="267",
(6,0) *+{345} ="345",
(6,1) *+{346} ="346",
(6,2) *+{347} ="347",
(6,3) *+{348} ="348",
(7,1) *+{356} ="356",
(7,2) *+{357} ="357",
(7,3) *+{358} ="358",
(8,2) *+{367} ="367",
(8,3) *+{368} ="368",
(9,3) *+{378} ="378",
(9.5,0) *+{456} ="456",
(9.5,1) *+{457} ="457",
(9.5,2) *+{458} ="458",
(9.5,3) *+{459} ="459",
(10.5,1) *+{467} ="467",
(10.5,2) *+{468} ="468",
(10.5,3) *+{469} ="469",
(11.5,2) *+{478} ="478",
(11.5,3) *+{479} ="479",
(12.5,3) *+{489} ="489",
(13,0) *+{567} ="567",
(13,1) *+{568} ="568",
(13,2) *+{569} ="569",
(13,3) *+{56X} ="56X",
(14,1) *+{578} ="578",
(14,2) *+{579} ="579",
(14,3) *+{57X} ="57X",
(15,2) *+{589} ="589",
(15,3) *+{58X} ="58X",
(16,3) *+{59X} ="59X",
(16.5,0) *+{123} ="678",
(16.5,1) *+{124} ="679",
(16.5,2) *+{125} ="67X",
(16.5,3) *+{ } ="67Y",
"123", {\ar"124"},
"124", {\ar"125"},
"124", {\ar"134"},
"125", {\ar"135"},
"134", {\ar"135"},
"134", {\ar"234"},
"135", {\ar"145"},
"135", {\ar"235"},
"145", {\ar"245"},
"234", {\ar"235"},
"235", {\ar"245"},
"235", {\ar"236"},
"236", {\ar"246"},
"246", {\ar"256"},
"245", {\ar"246"},
"236", {\ar"237"},
"246", {\ar"247"},
"256", {\ar"257"},
"237", {\ar"247"},
"247", {\ar"257"},
"257", {\ar"267"},
"345", {\ar"346"},
"346", {\ar"347"},
"347", {\ar"348"},
"356", {\ar"357"},
"357", {\ar"358"},
"367", {\ar"368"},
"346", {\ar"356"},
"347", {\ar"357"},
"357", {\ar"367"},
"348", {\ar"358"},
"358", {\ar"368"},
"368", {\ar"378"},
"456", {\ar"457"},
"457", {\ar"458"},
"458", {\ar"459"},
"467", {\ar"468"},
"468", {\ar"469"},
"478", {\ar"479"},
"457", {\ar"467"},
"458", {\ar"468"},
"468", {\ar"478"},
"459", {\ar"469"},
"469", {\ar"479"},
"479", {\ar"489"},
"567", {\ar"568"},
"568", {\ar"569"},
"569", {\ar"56X"},
"578", {\ar"579"},
"579", {\ar"57X"},
"589", {\ar"58X"},
"568", {\ar"578"},
"569", {\ar"579"},
"579", {\ar"589"},
"56X", {\ar"57X"},
"57X", {\ar"58X"},
"58X", {\ar"59X"},
"678", {\ar"679"},
"679", {\ar"67X"},
"245", {\ar"345"},
"246", {\ar"346"},
"256", {\ar"356"},
"247", {\ar"347"},
"257", {\ar"357"},
"267", {\ar"367"},
"356", {\ar"456"},
"357", {\ar"457"},
"367", {\ar"467"},
"358", {\ar"458"},
"368", {\ar"468"},
"378", {\ar"478"},
"467", {\ar"567"},
"468", {\ar"568"},
"478", {\ar"578"},
"469", {\ar"569"},
"479", {\ar"579"},
"489", {\ar"589"},
"578", {\ar"678"},
"579", {\ar"679"},
"57X", {\ar"67X"},
\end{xy}
}
\end{equation*}
\end{Eg}

\begin{Eg}\label{egSelfinj}
    Let $d = 2$, $n=4$ and $\ell = 3$. 
    Then the Gabriel quiver $Q_{4, 3}^{(2)}$ of the self-injective $2$-Nakayama algebra $A_{4, 3}^{(2)}$ is given as follows.
    
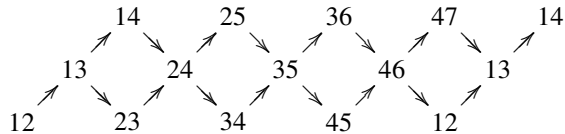
\begin{equation*}
        \begin{xy}
            0;<20pt,0cm>:<20pt,20pt>::
            (0,0) *+{12} ="12",
            (0,1) *+{13} ="13",
            (0,2) *+{14} ="14",
            (2,0) *+{23} ="23",
            (2,1) *+{24} ="24",
            (2,2) *+{25} ="25",
            (4,0) *+{34} ="34",
            (4,1) *+{35} ="35",
            (4,2) *+{36} ="36",
            (6,0) *+{45} ="45",
            (6,1) *+{46} ="46",
            (6,2) *+{47} ="47",
            (8,0) *+{12} ="56",
            (8,1) *+{13} ="57",
            (8,2) *+{14} ="58",
            "12", {\ar "13"},
            "13", {\ar "14"},
            "23", {\ar "24"},
            "24", {\ar "25"},
            "34", {\ar "35"},
            "35", {\ar "36"},
            "45", {\ar "46"},
            "46", {\ar "47"},
            "56", {\ar "57"},
            "57", {\ar "58"},
            "13", {\ar "23"},
            "14", {\ar "24"},
            "24", {\ar "34"},
            "25", {\ar "35"},
            "35", {\ar "45"},
            "36", {\ar "46"},
            "46", {\ar "56"},
            "47", {\ar "57"},
        \end{xy}
    \end{equation*}
    And the distiguished $2\mathbb{Z}$-cluster tilting subcategory $\mathcal{M}_{4,3}^{(2)}$ is given below.
    \begin{equation*}
{\small
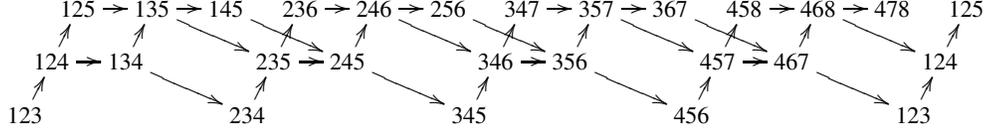
    \begin{xy}
        0;<28pt,0cm>:<10pt,20pt>::
        (0,0) *+{123} ="123",
        (0,1) *+{124} ="124",
        (0,2) *+{125} ="125",
        (1,1) *+{134} ="134",
        (1,2) *+{135} ="135",
        (2,2) *+{145} ="145",
        (3,0) *+{234} ="234",
        (3,1) *+{235} ="235",
        (3,2) *+{236} ="236",
        (4,1) *+{245} ="245",
        (4,2) *+{246} ="246",
        (5,2) *+{256} ="256",
        (6,0) *+{345} ="345",
        (6,1) *+{346} ="346",
        (6,2) *+{347} ="347",
        (7,1) *+{356} ="356",
        (7,2) *+{357} ="357",
        (8,2) *+{367} ="367",
        (9,0) *+{456} ="456",
        (9,1) *+{457} ="457",
        (9,2) *+{458} ="458",
        (10,1) *+{467} ="467",
        (10,2) *+{468} ="468",
        (11,2) *+{478} ="478",
        (12,0) *+{123} ="567",
        (12,1) *+{124} ="568",
        (12,2) *+{125} ="569",
        "123", {\ar"124"},
        "124", {\ar"125"},
        "134", {\ar"135"},
        "124", {\ar"134"},
        "125", {\ar"135"},
        "135", {\ar"145"},
        "234", {\ar"235"},
        "235", {\ar"236"},
        "245", {\ar"246"},
        "235", {\ar"245"},
        "236", {\ar"246"},
        "246", {\ar"256"},
        "345", {\ar"346"},
        "346", {\ar"347"},
        "356", {\ar"357"},
        "346", {\ar"356"},
        "347", {\ar"357"},
        "357", {\ar"367"},
        "456", {\ar"457"},
        "457", {\ar"458"},
        "467", {\ar"468"},
        "457", {\ar"467"},
        "458", {\ar"468"},
        "468", {\ar"478"},
        "567", {\ar"568"},
        "568", {\ar"569"},
        "134", {\ar"234"},
        "135", {\ar"235"},
        "145", {\ar"245"},
        "245", {\ar"345"},
        "246", {\ar"346"},
        "256", {\ar"356"},
        "356", {\ar"456"},
        "357", {\ar"457"},
        "367", {\ar"467"},
        "467", {\ar"567"},
        "468", {\ar"568"},
    \end{xy}}
\end{equation*}
\end{Eg}

\section{Singularity category of higher Nakayama algebras}

\subsection{Resolution quiver for $A_{\underline{\ell}}^{(d)}$}\label{section_resqui}
In this section, we introduce the resolution quiver of a given higher Nakayama algebra,
which is defined combinatorically but reflects certain homological properties of the algebra.

For fixed positive integers $n$ and $d$,
let $A = A_{\underline{\ell}}^{(d)}$ be the $d$-Nakayama algebra of type $\widetilde{\mathbb{A}}_{n-1}$ with Kupisch series $\underline{\ell} = (\ell_1, \ldots, \ell_n)$.
Recall that $\mathcal{M}_{\underline{\ell}}^{(d)} = \add \{M(x) \mid x\in os_{\ell_{\infty}}^{d+1}\}$ is the distinguished $d\mathbb{Z}$-cluster tilting subcategory of $\modn A$ as we described in Section \ref{HighNakAlg}.
Let $\mathcal{P} = \add \{ M(x) \in \mathcal{M}_{\underline{\ell}}^{(d)}  \mid x_{d+1} - x_1 + 1 = \ell_{x_{d+1}} + d\}$ be the additive category consisting of projective objects in $\mathcal{M}_{\underline{\ell}}^{(d)}$.

Define the map
\begin{equation*}
    f: \mathbb{Z} \rightarrow \mathbb{Z}
\end{equation*}
by $f(i) = i  - \ell_i - d + 1$. 
\begin{Rem}\label{f_periodic} The map $f$ has the following properties.
    \begin{itemize}
        \item [(i)] We have that
        $f(i+n) = f(i)+n$ since $\ell_i = \ell_{i+n}$. 
        Indeed,
        \begin{equation*}
            f(i + n) = (i+n)-\ell_{i+n}-d+1 = i-\ell_{i}-d+1 +n = f(i)+n.
        \end{equation*}
        \item [(ii)] $f$ is non-decreasing. 
        We know that $i - \ell_i \geq (i - 1) - \ell_{i - 1}$ since $\ell_i \leq \ell_{i-1} + 1$. 
        This implies $f(i) \geq f(i-1)$.
    \end{itemize}
\end{Rem}
By Remark \ref{f_periodic} $(i)$, $f$ induces a function 
\begin{align*}
    \overline{f}: \{1, 2, \ldots, n\} \rightarrow \{1, 2, \ldots, n\}
\end{align*}
such that $ \overline{f}(i) \equiv i - \ell_{i} - d + 1\mod  n $.
\begin{Def}
    The resolution quiver $R(A)$ of $A$ is defined as follows:
    \begin{itemize}
        \item vertices: $1, 2, \ldots, n$,
        \item arrows: $i\rightarrow \overline{f}(i)$ for each vertex $i$.
    \end{itemize}
\end{Def}

The notion resolution quiver is justified in the sense that $\overline{f}$, which is a normalization of $f$, detects periodic $\Omega^d$-orbits of $\mathcal{M}_{\underline{\ell}}^{(d)}$ as shown in Proposition \ref{prop_f_detects_syzygyorbit}.
Moreover, when $d=1$, our definition coincides with the definition of the resolution quiver for usual Nakayama algebras in \cite{Sh15}, which was originally introduced in \cite{Rin13}.

\begin{Prop}\label{prop_f_detects_syzygyorbit} Let $x = (x_1, \ldots, x_{d+1})\in os_{\ell_{\infty}}^{d+1}$. Then $M(x) \in \mathcal{P}$ if and only if $x_1 = f(x_{d+1})$. 
Otherwise $\Omega^d (M(x)) = M(f(x_{d+1}), x_1, \ldots, x_d)$. 
Moreover, we have the following exact sequence
\begin{equation*}
    \xymatrix@=1em@R=1em{0\ar[r] & \Omega^dM(x)\ar[r] & P_d\ar[r] & \cdots\ar[r] & P_1\ar[r] & M(x)\ar[r] & 0}
\end{equation*}
where $P_i = M(f(x_{d+1}), x_1, \ldots, x_{i-1}, x_{i+1}, \ldots, x_{d+1})$ is projective for $1\leq i\leq d$.
\end{Prop}
\begin{proof}
    $x = (x_1, \ldots, x_{d+1})\in os_{\underline{\ell}}^{d+1}$ implies that $f(x_{d+1})\leq x_1$.
     As shown in section \ref{CT_Lift},
    $M(x) = F_{\ast}\widehat{M}(x)$ with $\widehat{M}(x)\in \mathcal{M}_{\ell_{\infty}}^{(d)}$. 
    Since $F_{\ast}$ is left adjoint to an exact functor, $F_{\ast}$ preserves projective modules.
    Thus $M(x)$ is projective if and only if $\widehat{M}(x)$ is projective if and only if $x_1 = f(x_{d+1})$ by Proposition \ref{Prop_modcomb} $(iii)$.

    In the case $f(x_{d+1}) < x_1$, we have 
    the beginning of the minimal projective resolution of $\widehat{M}(x)$ by Proposition \ref{Prop_modcomb} $(v)$
    \begin{equation*}
        \xymatrix@C=1em@R=1em{
            0\ar[r] & \Omega^d(\widehat{M}(x))\ar[r] & Q_{d}\ar[r] & \cdots\ar[r] & Q_{1}\ar[r] & \widehat{M}(x)\ar[r] & 0
        }
    \end{equation*}
    with $Q_{i} = \widehat{M}(f(x_{d+1}), x_1, \ldots, x_{i-1}, x_{i+1}, \ldots, x_{d+1})$ for $1\leq i\leq d$
    and
    $\Omega^d(\widehat{M}(x)) =$ \\$ \widehat{M}(f(x_{d+1}), x_1, \ldots, x_d)$.
    We shall apply $F_{\ast}$ which is exact and preserves projectivity to obtain the following exact sequence 
    \begin{equation*}
        \xymatrix@C=1em@R=1em{
        0\ar[r] & \Omega^d(M(x))\ar[r] & P_{d}\ar[r] & \cdots\ar[r] & P_{1}\ar[r] & M(x)\ar[r] & 0
        }
    \end{equation*}
    where $P_{i} = F_{\ast}Q_{i}$ for $1\leq i\leq d$.
    The claim follows.
\end{proof}

\begin{Eg}\label{egResQui}
    Let $n = 5$ and $\underline{\ell} = (3,4,4,4,4)$. When $d = 1$, we have that $\overline{f}(i) \equiv i - \ell_{i} \mod 5$. 
    Thus the resolution quiver $R({A_{\underline{\ell}}^{(1)}})$ is given by 
    \begin{equation*}
        \xymatrix{
           4\ar[r] & 5\ar[r] & 1\ar[r] & 3\ar@/^ 1pc/[lll] & 2\ar[l]
        }.
    \end{equation*}\\[4pt]
    When $d = 2$, $\overline{f}(i) \equiv i - \ell_{i} - 1 \mod 5$.
    Then we have the resolution quiver of $R({A_{\underline{\ell}}^{(2)}})$ as follows.
    \begin{equation*}
        \xymatrix{
           1\ar[r] & 2\ar@(r,d)[] & 3\ar@(r,d)[] & 4\ar@(r,d)[] & 5\ar@(r,d)[] 
        }
    \end{equation*}\\[6pt]
    When $d = 3$, $\overline{f}(i) \equiv i - \ell_{i} - 2 \mod 5$.
    Then we have the resolution quiver of $R({A_{\underline{\ell}}^{(3)}})$ as follows.
    \begin{equation*}
        \xymatrix{
          5\ar[r] & 4\ar[r] & 3\ar[r] & 2\ar[r] & 1\ar@(r,d)[]
        }
    \end{equation*}
\end{Eg}\medskip

We wish to capture homological information of singularity categories of higher Nakayama algebras using their resolution quivers. 
To do this, we explore more properties of the functions $f$ and $\overline{f}$.

Let $J = \{i \mid 1\leq i\leq n, \overline{f}^{N}(i) = i \text{ for some } N\in \mathbb{N}_{>0}\}$ and $I = J + n\mathbb{Z}$.
For $i\in \mathbb{Z}$, denote by $\overline{i}\in \{1, 2, \ldots, n\}$ the representative of $i$ such that $i = \overline{i} + n\mathbb{Z}$. 
Note that $i\in I$ if and only if $\overline{i}\in J$.
\begin{Prop}\label{Prop_f}
The following statements hold.
    \begin{itemize}
        \item [(i)] $\overline{f}|_{J}: J\rightarrow J$ is bijective.
        \item [(ii)] $f|_{I}: I\rightarrow I$ is a bijection. Moreover, there exist $s, t\in \mathbb{N}$ such that $f^s (i) = i+tn$ for all $i\in I$.
        \item [(iii)] There exists some $N\in \mathbb{N}$ such that for all $i\in \mathbb{Z}$, $f^{N}(i)\in I$.
    \end{itemize}
\end{Prop}
\begin{proof}
    (i) Since $\overline{f}(J) = J$ by definition and $J$ is a finite set,
    the statement follows.
    
    (ii) For $i\in I$, there is an $ N\in \mathbb{N}$ such that $\overline{f}^N(\overline{i}) = \overline{i}$.
    So we have $i = f^N(i) + pn = f^N(i+pn)$ for some $p\in \mathbb{Z}$.
    Thus the surjectivity of $f|_{I}$ follows.
    For injectivity, suppose $f(i) = f(j)$ with $i,j\in I$.
    Then $\overline{f}(\overline{i}) = \overline{f}(\overline{j})$ which implies that $j = i + qn$ for some $q\in \mathbb{Z}$.
    But $f(i) = f(j) = f(i+qn) = f(i) + qn$ forces $q = 0$.
    Therefore, $i = j$.
    
    If $|J| = 1$, then we take $s = t\in \mathbb{N}$.
    Otherwise, consider $i, j$ which are adjacent in $I$.
    Since $f$ is non-decreasing and bijective on $I$,
    $f(i), f(j)$ are also adjacent.
    By induction, $f^m(i), f^m(j)$ are adjacent for any $m\in \mathbb{N}$.
    
    As before, there are $M, N\in \mathbb{N}$ such that $f^M(i) = i + pn$ and $f^N(j) = j + qn$ for some $p, q\in \mathbb{Z}$. 
    Hence we can choose $s\in \mathbb{N}$ such that $f^s(i) = i + t_in$ and $f^s(j) = j + t_j n$ for some $t_i, t_j\in \mathbb{N}$.
    Since $f^s(i), f^s(j)$ are adjacent,
    we have $t_i = t_j$.
    Therefore, we can choose $s, t\in \mathbb{N}$ such that $f^s(i) = i + tn$ for all $i\in I$.
    
    (iii) Let $i\in \mathbb{Z}$.
    Then there is $ N_i\in \mathbb{N}$ such that $\overline{f}^{N_i}(\overline{i}) \in J$.
    So $f^{N_i}(\overline{i}) \in I$ and $f^{N_i}(i)\in I$.
    Take $N = N_1N_2\cdots N_n$. 
    Then $f^N(i)\in I$ for all $1\leq i\leq n$ and thus for all $i\in \mathbb{Z}$.
\end{proof}
Let $J' = \{1, 2, \ldots, n'\}$ with $n' = |J|$. Denote by $\iota: \mathbb{Z} \rightarrow I$ the unique order preserving bijection.
Let $\ell_{\infty}' = (\ldots, \ell'_{-1}, \ell'_0, \ldots)$ where  $\ell'_k = |[f(\iota(k)), \iota(k)] \cap I| - d$. The following proposition shows $\ell_{\infty}'$ is a series with constant values.
\begin{Prop}\label{Prop_NewKup}
    Notations are as above. There exists an integer $\ell'$ such that $\ell_k' = \ell'$ for all $k\in \mathbb{Z}$.
\end{Prop}
\begin{proof}
    Firstly we show that $\ell_k' \leq \ell_{k-1}' + 1$ for all $k\in \mathbb{Z}$.
    Since $f$ is non-decreasing, we have the following inequalities
    \begin{equation*}\begin{split}
        \ell_k' & = |[f(\iota(k)), \iota(k)] \cap I| - d \\
        & \leq |[f(\iota(k-1)), \iota(k)] \cap I| -d 
        \\ & = |[f(\iota(k-1)), \iota(k-1)] \cap I| - d + 1 \\
        & = \ell_{k-1}' +1. \end{split}
    \end{equation*}
    We claim that $\ell_k' \leq \ell_{k-1}'$. 
    Otherwise, $\ell_k' = \ell_{k-1}'+1$.
    Then 
    \begin{equation*}
        |[f(\iota(k)), \iota(k-1)]\cap I| + 1 = |[f(\iota(k)), \iota(k)]\cap I| = |[f(\iota(k-1)), \iota(k-1)]\cap I|+1.
    \end{equation*}
    Since $f$ is bijective on $I$, this implies $f(\iota(k-1)) = f(\iota(k))$ and thus $k = k-1$ which is impossible. 
    Therefore, 
    \begin{equation*}
        \ell_k' \leq \ell_{k-1}' \leq \cdots \leq \ell_{k-n'}' = \ell_k'
    \end{equation*}
    which implies that $\ell_k' = \ell_{k-1}'$ for all $k\in \mathbb{Z}$.
\end{proof}

\begin{Eg}
    Let $n = 5$, $d = 2$ and $\underline{\ell} = (3,4,4,4,4)$. 
    By Example \ref{egResQui}, 
    we have that $J = \{2, 3, 4, 5\}$ and $I = J + 5\mathbb{Z}$.
    Recall that $f(i) = i - \ell_{i} - 1$.
    Then 
    \begin{equation*}
        f(1) = -3, f(2) = -3, f(3) = -2, f(4) = -1, f(5) = 0.
    \end{equation*}
    Following our construction,
    $J' = \{1, 2, 3, 4\}$ and 
    \begin{equation*}
        \iota(1) = 2, \iota(2) = 3, \iota(3) = 4, \iota(4) = 5.
    \end{equation*}
    Hence 
    \begin{align*}
        \ell'_1 & = |[f(\iota(1)), \iota(1)] \cap I| -2 = |[-3, 2] \cap I| - 2 = 3,\\
        \ell'_2 & = |[f(\iota(2)), \iota(2)] \cap I| -2 = |[-2, 3] \cap I| - 2 = 3,\\
        \ell'_3 & = |[f(\iota(3)), \iota(3)] \cap I| -2 = |[-1, 4] \cap I| - 2 = 3,\\
        \ell'_4 & = |[f(\iota(4)), \iota(4)] \cap I| -2 = |[0, 5] \cap I| - 2 = 3.
    \end{align*}
    We have that $\ell'_{\infty} = (\ldots, 3, 3, 3, \ldots)$.
\end{Eg}

\begin{Rem}\label{f'}
    The function $\iota: \mathbb{Z} \rightarrow I$ is actually a relabelling of $I$. 
    We have the following commutative diagram
    \begin{equation*}
        \xymatrix{
            \mathbb{Z}\ar[r]^{\iota}\ar[d]_{f'} & I\ar[d]^{f|_I} \\
            \mathbb{Z}\ar[r]^{\iota} & I,
        }
    \end{equation*}
    where $f' = \iota^{-1}\circ f|_I\circ \iota$. 
    Since $f|_I: I \rightarrow I$ is bijective we have that $f'$ is bijective.
    Moreover, 
    \begin{equation*}
        \ell'_i = i - f'(i)  - d + 1\text{ for } i\in \mathbb{Z}.
    \end{equation*}
    As can be seen from the above formula, we obtained $\ell'_{\infty}$ by forcing $f'$ to play the role of $f|_I$.
    Observe that when $\ell' \geq 2$, $\ell_{\infty}'$ is a connected Kupisch series with constant values.
    The case $\ell'\leq 1$ is addressed below.
\end{Rem}

We extend $f$ to $os_{\ell_{\infty}}^{d+1}\cup \{0\}$ in the following way.
For $x = (x_1, \ldots, x_{d+1})\in os_{\ell_{\infty}}^{d+1}$,
define $f(x) = (f(x_1), \ldots, f(x_{d+1}))$ if $f(x_1) < f(x_2) < \cdots f(x_{d+1}) < x_1$ and $f(x) = 0$ otherwise. 
Further define $f(0) = 0$.
Note that if $f(x) \neq 0$, then $f(x)\in os_{\ell_{\infty}}^{d+1}$ since $f(x_{d+1}) < x_1$ implies $f^2(x_{d+1}) \leq f(x_1)$.

\begin{Lem}\label{lem_syzygy_f}
Let $x\in os_{\ell_{\infty}}^{d+1}$.
Then $\projdim M(x) \leq d^2$ if $f(x) = 0$.
Otherwise $\Omega^{d(d+1)}M(x) = M(f(x))$.
In particular, if $f^N(x) = 0$ for some $N\geq 1$ then $\projdim M(x) < \infty$.
\end{Lem}
\begin{proof}
    If $f(x_{d+1}) = x_1$, then $M(x)$ is projective. Thus $f(x) = 0$ and the statement holds.
    If $f(x_{d+1}) < x_1$, then $\Omega^dM(x) = M(f(x_{d+1}), x_1, \ldots, x_d)$ by Proposition \ref{prop_f_detects_syzygyorbit}. 
    Iterating this process, we find either that some $\Omega^{di}M(x)$ is projective where $1\leq i\leq d$ and $f(x) = 0$ or $f(x)\neq 0$ and $\Omega^{d(d+1)}M(x) = M(f(x))$.
    By induction, the second claim follows.
\end{proof}

\begin{Prop}
    If $\ell' \leq 1$, then $D_{sg}(A) = 0$.
\end{Prop}
\begin{proof}
    We claim that for all $x\in os_{\ell_{\infty}}^{d+1}$, $f^s(x) = 0$ for some $s \gg 0$.
    By Proposition \ref{Prop_f} $(iii)$, 
    there exists $N \gg 0$ such that $f^N(x_i)\in I$ for all $1\leq i\leq d+1$.
    In the case $\ell' \leq 0$, we have $f^{N+1}(x) = 0$. If not, then $(f^{N+1}(x_{d+1}), f^N(x_1), \ldots, f^N(x_d))\in os_{\ell_{\infty}}^{d+1}$ which is impossible.
    When $\ell' = 1$, we have $f^{N+1}(x_{d+1}) = f^N(x_1)$. 
    In other words, $M(f^N(x))$ is projective thus $f^{N+1}(x) = 0$.

    In both cases, $\projdim M(x) < \infty$ by Lemma \ref{lem_syzygy_f}.
    Thus $M(x)\in K^b(\proj A)$.
    As we have seen in the proof of Theorem \ref{Thm_main} that $D^b(\modn A) = \tri (\mathcal{M}_{\ell_{\infty}}^{(d)})$.
    Therefore, $D^b(\modn A) = K^b(\proj A)$ and $D_{sg}(A) = 0$.
\end{proof}

\begin{Rem}
    From now on we assume $\ell' \geq 2$.
    Then $\ell'_{\infty}$ is a connected Kupisch series with constant values.
\end{Rem}

\subsection{Singularity category of $A_{\underline{\ell}}^{(d)}$}
In this section, we will define a self-injective higher Nakayama algebra $A'$ by the Kupisch series obtained from section \ref{section_resqui}.
By identifying $A'$ with an idempotent subalgebra $B$ of $A$ and applying Theorem \ref{Thm_main},
we obtain a singular equivalence between $A$ and $B$.
Therefore the singularity category of a $d$-Nakayama algebra is triangulated equivalent to the stable module category of a selfinjective $d$-Nakayama algebra.

Let $A' = A_{n', \ell'}^{(d)}$ be the selfinjective $d$-Nakayama algebra of type $\widetilde{\mathbb{A}}_{n'-1}$ with Kupisch series $(\ell', \ldots, \ell')$ and 
\begin{equation*}
    \mathcal{M}_{n', \ell'}^{(d)} = F_{\ast}\mathcal{M}_{\ell_{\infty}'}^{(d)} = \add \{M(x) = F_{\ast}\widehat{M}(x) \mid x\in os_{\ell_{\infty}'}^{d+1}\}
\end{equation*}
be the distinguished $d\mathbb{Z}$-cluster tilting subcategory of $\modn A'$.

Let 
\begin{equation*}
    os_I^{d+1} = \{x = (x_1, \ldots, x_{d+1})\in os_{\ell_{\infty}}^{d+1} \mid  x_i\in I \text{ for all } 1\leq i\leq d+1\}
\end{equation*}
be the set with ordered squences whose coordinates are in $I$.
Let 
\begin{equation*}
    \mathcal{W} = \add \{M(x) \in \mathcal{M}_{\underline{\ell}}^{(d)} \mid x\in os_I^{d+1}\} \subset \mathcal{M}_{\underline{\ell}}^{(d)}
\end{equation*} 
and $\mathcal{P}_I = \mathcal{W} \cap \mathcal{P}$, i.e. the full subcategory of projective objects in $\mathcal{W}$.

Recall that $\iota: \mathbb{Z}\rightarrow I$ is the unique order-preserving bijection.
Now we extend $\iota$ to a bijection preserving the relation $\preccurlyeq$ as follows
\begin{equation*}
    \iota: os_{\ell'_{\infty}}^{d+1}\rightarrow os_I^{d+1}
\end{equation*}
where $\iota(x) = (\iota(x_1), \ldots, \iota(x_{d+1}))$ for $x = (x_1, \ldots, x_{d+1})\in os_{\ell'_{\infty}}^{d+1}$.

We define a $k$-linear functor induced by $\iota$ as follows.
\begin{align*}
    \iota^{\ast}: \mathcal{M}_{n', \ell'}^{(d)} & \rightarrow \mathcal{W}\\
    M(x) & \mapsto M(\iota(x)) \\
    [M(x)\xrightarrow[]{f_{yx}^i} M(y)] & \mapsto [ M(\iota(x))\xrightarrow[]{f_{\iota(y)\iota(x)}^i} M(\iota(y))].
\end{align*}
Observe that $\iota^{\ast}(Id_{M(x)}) = \iota^{\ast}(f_{xx}^0) = f_{\iota(x)\iota(x)}^0 = Id_{\iota^{\ast}(M(x))}$ and $\iota^{\ast}(f_{zy}^jf_{yx}^i) = \iota^{\ast}(f_{zy}^j)\iota^{\ast}(f_{yx}^i)$ since $\iota$ preserves the relation $\preccurlyeq$.
Moreover $\iota^{\ast}$ is fully faithful since it maps a $k$-basis of $\Hom_{A'}(M(x), M(y))$ to a $k$-basis of $\Hom_{A}(M(\iota(x)), M(\iota(y)))$. 
For an indecomposable object $M(z)\in \mathcal{W}$, 
we have that $\iota^{-1}(z)\in os_{\ell_{\infty}'}^{d+1}$ since $\iota$ is bijective. 
Thus $M(\iota^{-1}(z))\in \mathcal{M}_{n', \ell'}^{(d)}$ and $\iota^{\ast}(M(\iota^{-1}(z))) = M(z)$.
This implies that $\iota^{\ast}$ is dense.
Therefore, $\iota^{\ast}$ is a $k$-linear equivalence.

\begin{Prop}\label{iotastarisequiv}
     We have a $k$-linear equivalence $\iota^{\ast}: \mathcal{M}_{n',\ell'}^{(d)} \rightarrow \mathcal{W}$.
     Moreover, when restricted to projective objects,
     we obtain the equivalence 
     $\iota^{\ast}: \add A' \rightarrow \add P$ where $P$ is a basic additive generator of $\mathcal{P}_I$.
\end{Prop}

\begin{proof}
    By Proposition \ref{prop_f_detects_syzygyorbit}, 
    $M(x) \in \mathcal{M}_{n',\ell'}^{(d)}$ is projective if and only if $x_1 = f'(x_{d+1})$.
    Since $f' = \iota^{-1}\circ f|I\circ \iota$ by Remark \ref{f'}, 
    this is equivalent to $\iota(x_1) = f\circ \iota(x_{d+1})$, 
    which is fulfilled if and only if $\iota^{\ast}(M(x))$ is projective.
    Therefore, $\iota^{\ast}$ restricts to an equivalence $\add A'\xrightarrow{\sim} \add P$.
\end{proof}

We denote by $B = \End_A(P)$ the endomorphism algebra of $P$.
Since $P$ is a basic projective $A$-module, there exists an idempotent $e\in A$ such that $P = eA$ and $B = eAe$.
In other words, $B$ is an idempotent subalgebra of $A$.

We have the canonical functor 
\begin{equation*}
    i_{\lambda} = -\otimes_B P : \modn B \rightarrow \modn A, N\mapsto N \otimes_B P,
\end{equation*}
which admits an exact right adjoint functor
\begin{equation*}
    i_{\rho} = \Hom_A(P, -):
    \modn A\rightarrow \modn B,
    M\mapsto Me.
\end{equation*}
Since $i_{\lambda}(B) = P$ and $i_{\rho}(P) = B$,
$i_{\lambda}, i_{\rho}$ restrict to an additive equivalence
\begin{equation*}
    \xymatrix{
      \add P \ar@/^/[rr]^{i_{\rho}} && \add B\ar@/^/[ll]^{i_{\lambda}}
    }.
\end{equation*}

\begin{Prop}\label{A'IsoB}
    We have that $B \cong A'$. That is, $B$ is a self-injective $d$-Nakyama algebra. 
\end{Prop}
\begin{proof}
   Write $A' = \bigoplus_{x\in X} M(x)$ where $X$ is the set of indices of indecomposable $A'$-projective modules.
   Consider the equivalence $i_{\rho}\circ \iota^{\ast}: \add A' \rightarrow \add B, M(x)\mapsto \Hom_A(P, M(\iota(x)))$.
   We have that 
   \begin{align*}
       A' & = \End_{A'}(A')\\
       & \cong \bigoplus_{x, y\in X}\Hom_{A'}(M(x), M(y))\\
       & \cong \bigoplus_{x, y\in X}\Hom_{A}(M(\iota(x)), M(\iota(y)))\\
       & \cong \bigoplus_{x, y\in X}\Hom_{B}(\Hom_A(P, M(\iota(x))),\Hom_A(P, M(\iota(y))))\\
       & \cong \Hom_{B}(\Hom_A(P, \bigoplus_{x\in X}M(\iota(x))),\Hom_A(P, \bigoplus_{y\in X}M(\iota(y))))\\
       & \cong \End_{B}(\Hom_A(P, P))\\
       & = B.
   \end{align*}
\end{proof}

We use the above isomorphism to identify $B$ with $A'$.
In the same way we identify $\add B = \add A'$ and the distinguished $d\mathbb{Z}$-cluster tilting subcategory of $B$ with $\mathcal{M}_{n', \ell'}^{(d)}$.

\begin{Prop}\label{Prop_prepForWide}
    For a nonprojective module $M(x)\in \mathcal{W}$,
        there is a positive integer $s$ and an exact sequence 
            \begin{equation*}
                \xymatrix@C=1em@R=1em{
            (\ast) ~~~   0\ar[r] & M(x)\ar[r] & P_{sd(d+1)}\ar[r] & \cdots\ar[r] & P_{1}\ar[r] & M(x)\ar[r] & 0
                }
            \end{equation*}
            with $P_{i}\in \mathcal{P}_I$.
\end{Prop}
\begin{proof}
     Since $f(I)\subset I$ and $x\in os_I^{d+1}$,
     the sequence 
     \begin{equation*}
         \xymatrix@C=1em@R=1em{
         0\ar[r] & \Omega^dM(x)\ar[r] & P_d\ar[r] & \cdots\ar[r] & P_1\ar[r] & M(x)\ar[r] & 0
         }
     \end{equation*}
     from Proposition \ref{prop_f_detects_syzygyorbit} lies in $\mathcal{W}$.
     In particular, $P_i\in \mathcal{P}_I$ for  $1\leq i\leq d$.
     Moreover, since $f|_I : I\rightarrow I$ is bijective by Proposition \ref{Prop_f} $(ii)$,
     we get that $\Omega^dM(x) = M(f(x_{d+1}), x_1, \ldots, x_d)\in \mathcal{W}$ is not projective.
     Indeed, $x_d < x_{d+1}$ implies $f(x_d) < f(x_{d+1})$.
     By iteration, we get 
     \begin{equation*}
         \xymatrix@C=1em@R=1em{
         0\ar[r] & \Omega^{dr}M(x)\ar[r] & P_{dr}\ar[r] & \cdots\ar[r] & P_1\ar[r] & M(x)\ar[r] & 0
         }
     \end{equation*}
     for any $r\geq 1$.
     By Lemma \ref{lem_syzygy_f}, we have that $\Omega^{d(d+1)s}M(x) = M(f^s(x))$ for all $s\geq 1$.
    By Proposition \ref{Prop_f} $(ii)$, there exists $s, t\in \mathbb{N}$ such that $f^s(i) = i+tn$ for all $i\in I$. Then
     $$\Omega^{d(d+1)s}(M(x)) = M(f^s(x)) \cong M(\sigma^t(x)) \cong M(x).$$
     In particular, we have the following exact sequence
    \begin{equation*}
        \xymatrix@C=1em@R=1em{
            0\ar[r] & M(x)\ar[r] & P_{d(d+1)s}\ar[r] & \cdots\ar[r] & P_{1}\ar[r] & M(x)\ar[r] & 0
        }
    \end{equation*}
    with $P_i\in \mathcal{P}_I$ for all $1\leq i\leq d(d+1)s$.
\end{proof}

\begin{Prop}\label{Prop_prepForWide2}
     There is a $k$-linear equivalence $i_{\rho}: \mathcal{W} \xrightarrow[]{\sim} \mathcal{M}_{n' ,\ell'}^{(d)}$. 
\end{Prop}
\begin{proof}
    By our identification $\add B \cong \add A'$ given by Proposition \ref{A'IsoB}, it follows that $i_{\rho}\circ \iota^{\ast}$ restricted to $\add B$ is isomorphic to $ Id_{\add B}$.
    Hence, for $M(x), M(y)\in \add P$, $i_{\rho}(M(x)) = M(\iota^{-1}(x))$ and $i_{\rho}(f_{yx}^i) = f_{\iota^{-1}(y)\iota^{-1}(x)}^i$ for $f_{yx}^i: M(x)\rightarrow M(y)$.
    For a nonprojective indecomposable object $M(x)\in \mathcal{W}$ with $x = (x_1, \ldots, x_{d+1})$, we take the minimal projective presentation of $M(x)$
    \begin{equation*}
        M(x^1)\rightarrow M(x^0)\rightarrow M(x)\rightarrow 0
    \end{equation*}
    where $x^0 = (f(x_{d+1}), x_2, \ldots, x_{d+1})$ and $x^1 = (f(x_{d+1}), x_1, x_3, \ldots, x_{d+1})$.
    Applying the exact functor $i_{\rho}$ to it gives us the minimal projective presentation of $i_{\rho}(M(x))$
    \begin{equation*}
        M(\iota^{-1}(x^1))\rightarrow M(\iota^{-1}(x^0))\rightarrow i_{\rho}(M(x))\rightarrow 0.
    \end{equation*}
    Thus $i_{\rho}(M(x)) = M(\iota^{-1}(x))$.
    For $f_{yx}^i: M(x)\rightarrow M(y)$, 
    we obtain, using the above identification, that $i_{\rho}(f_{yx}^i) = f_{\iota^{-1}(y)\iota^{-1}(x)}^i : M(\iota^{-1}( x))\rightarrow M(\iota^{-1}(y))$.
    Since $i_{\rho}$ is $k$-linear, it follows that $i_{\rho}$ is an additive equivalence, namely the quasi-inverse is given by $\iota^{\ast}$.
\end{proof}

\begin{Prop}\label{W_IsWide}
  The following statements hold.
    \begin{itemize}
        \item [(a)] $\mathcal{W}$ is a wide subcategory of     $\mathcal{M}_{\underline{\ell}}^{(d)}$.
        \item [(b)] $i_{\lambda}$ and $i_{\rho}$ restrict to a quasi-inverse equivalence of $d$-abelian categories
    \begin{equation*}
        \xymatrix{
          \mathcal{M}_{n', \ell'}^{(d)}\ar@/^/[rr]^{i_{\lambda}} && \mathcal{W}\ar@/^/[ll]^{i_{\rho}}.
         }
    \end{equation*}
    \end{itemize}
\end{Prop}
\begin{proof}
    We consider $P\in \mathcal{W}$.
    To apply Theorem \ref{Thm_WideSubcat},
    condition $(i)$ and $(ii)$ are trivial since $P_A$ is projective.
    Take a nonprojective object $M(x) \in \mathcal{W}$.
    By splicing the exact sequences $(\ast)$ in Proposition \ref{Prop_prepForWide}, we obtain an exact (periodic) $\add P$-resolution of $M(x)$.
    Proposition \ref{Prop_prepForWide2} verifies condition $(iv)$.
    Therefore $\mathcal{W}$ is a wide subcategory of $\mathcal{M}_{\underline{\ell}}^{(d)}$.
    Part (b) also follows from Theorem \ref{Thm_WideSubcat}.
\end{proof}

\begin{Eg}\label{egWideSubCat}
Let $n = 5$, $d = 2$ and $\underline{\ell} = (3,4,4,4,4)$. By Example \ref{egCTSubCat} and Example \ref{egResQui},
we have the Auslander-Reiten quiver of $\mathcal{W}$ as follows. Note that it is the same as the Auslander-Reiten quiver of $\mathcal{M}_{4,3}^{(2)}$ as in Example \ref{egSelfinj}.
\begin{equation*}
{\small
    \begin{xy}
        0;<28pt,0cm>:<10pt,20pt>::
        (0,0) *+{234} ="234",
        (0,1) *+{235} ="235",
        (0,2) *+{237} ="237",
        (1,1) *+{245} ="245",
        (1,2) *+{247} ="247",
        (2,2) *+{257} ="257",
        (3,0) *+{345} ="345",
        (3,1) *+{347} ="347",
        (3,2) *+{348} ="348",
        (4,1) *+{357} ="357",
        (4,2) *+{358} ="358",
        (5,2) *+{378} ="378",
        (6,0) *+{457} ="457",
        (6,1) *+{458} ="458",
        (6,2) *+{459} ="459",
        (7,1) *+{478} ="478",
        (7,2) *+{479} ="479",
        (8,2) *+{489} ="489",
        (9,0) *+{578} ="578",
        (9,1) *+{579} ="579",
        (9,2) *+{57X} ="57X",
        (10,1) *+{589} ="589",
        (10,2) *+{58X} ="58X",
        (11,2) *+{59X} ="59X",
        (12,0) *+{234} ="789",
        (12,1) *+{235} ="78X",
        (12,2) *+{237} ="78Y",
        "234", {\ar"235"},
        "235", {\ar"237"},
        "245", {\ar"247"},
        "235", {\ar"245"},
        "237", {\ar"247"},
        "247", {\ar"257"},
        "345", {\ar"347"},
        "347", {\ar"348"},
        "357", {\ar"358"},
        "347", {\ar"357"},
        "348", {\ar"358"},
        "358", {\ar"378"},
        "457", {\ar"458"},
        "458", {\ar"459"},
        "478", {\ar"479"},
        "458", {\ar"478"},
        "459", {\ar"479"},
        "479", {\ar"489"},
        "578", {\ar"579"},
        "579", {\ar"57X"},
        "589", {\ar"58X"},
        "579", {\ar"589"},
        "57X", {\ar"58X"},
        "58X", {\ar"59X"},
        "789", {\ar"78X"},
        "78X", {\ar"78Y"},
        "245", {\ar"345"},
        "247", {\ar"347"},
        "257", {\ar"357"},
        "357", {\ar"457"},
        "358", {\ar"458"},
        "378", {\ar"478"},
        "478", {\ar"578"},
        "479", {\ar"579"},
        "489", {\ar"589"},
        "589", {\ar"789"},
        "58X", {\ar"78X"},
    \end{xy}}
\end{equation*}
\end{Eg}

\begin{Thm}\label{Thm_main}
    The functor $i_{\lambda}: \modn B \rightarrow \modn A$ induces a singular equivalence between $B$ and $A$.
    More precisely,
    \begin{equation*}
        D_{sg}(i_{\lambda}): D_{sg}(B) \rightarrow D_{sg}(A)
    \end{equation*}
    is an triangule equivalence.
    In addition, $D_{sg}(i_{\lambda})$ restricts to an equivalence between $(d+2)$-angulated categories
    \begin{equation*}
        D_{sg}(i_{\lambda}): \underline{\mathcal{M}}_{n', \ell'}^{(d)} \rightarrow \underline{\mathcal{M}}_{\underline{\ell}}^{(d)}.
    \end{equation*}
\end{Thm}
\begin{proof}
    As in the proof of Proposition \ref{W_IsWide}, we can apply Theorem \ref{Thm_WideSubcat}.
    We verify the condition in Theorem \ref{Thm_WideSubcat} $(b)$. For $M(x)\in \mathcal{M}_{\underline{\ell}}^{(d)}$,
    by Proposition \ref{Prop_f} $(iii)$, there is an integer $ N\in \mathbb{N}$ such that $f^{N}(x)\in os_I^{d+1}$.
    This implies that $\Omega^{d(d+1)N}(M(x)) = M(f^{N}(x)) \in \mathcal{W}$ by Lemma \ref{lem_syzygy_f}.
    So Theorem \ref{Thm_WideSubcat} $(b)$ applies and 
    we have that $D_{sg}(i_{\lambda}): D_{sg}(B)\rightarrow D_{sg}(A)$ is a triangle equivalence.
    Since $\mathcal{M}_{\underline{\ell}}^{(d)}$ is a $d\mathbb{Z}$-cluster tilting subcategory of $\modn A$,
    by Theorem \ref{Thm_WideSubcat} $(c)$, 
    $D_{sg}(i_{\lambda})$ restricts to an equivalence between $(d+2)$-angulated categories $\underline{\mathcal{M}}_{n', \ell'}^{(d)}$ and $\underline{\mathcal{M}}_{\underline{\ell}}^{(d)}$.
\end{proof}

\begin{Cor}\label{corSingEquiToStableCat}
    The singularity category of $A$ is triangulated equivalent to the stable module category of $B$.
    More precisely,  $i_{\lambda}: \modn B\rightarrow \modn A$ induces an triangule equivalence
    \begin{equation*}
        D_{sg}(i_{\lambda}): \underline{\modn }B\rightarrow D_{sg}(A).
    \end{equation*}
\end{Cor}
\begin{proof}
     Since $B$ is self-injective, we have that $D_{sg}(B)\cong \underline{\modn }B$. 
     By Theorem \ref{Thm_main}, 
     the statement follows.
\end{proof}

\begin{Rem}
    We have the following commutative diagram
     \begin{equation*}
         \xymatrix@C=6em@R=2em{
             \modn B\ar[r]^{i_{\lambda}}\ar[d] &  \modn A\ar[d]\\
             \underline{\modn }B\ar[r]^{\underline{i_{\lambda}}}\ar[d]_{Id} & \underline{\modn }B\ar[d]^{q}\\
             \underline{\modn }B\ar[r]^{D_{sg}(i_{\lambda})} & D_{sg}(A).
         }
     \end{equation*}
     Hence $D_{sg}(i_{\lambda}) = q\circ \underline{i_{\lambda}}$. 
\end{Rem}

By \cite[Theorem A]{JKM22}, there is a bijective correspondence 
between the equivalence classes of pairs $(\mathcal{T}, c)$ with $\mathcal{T}$ an algebraic Krull-Schmidt triangulated category with finite dimensional morphism spaces and $c$ a basic $d\mathbb{Z}$-cluster tilting object of $\mathcal{T}$ and the equivalence classes of pairs $(\Lambda, I)$ with $\Lambda$ a basic twisted $(d+2)$-periodic self-injective algebra and $I$ an invertible $\Lambda$-bimodule, such that 
$I \cong \Omega_{\Lambda^e}^{d+2}(\Lambda)$ in $\underline{\modn }\Lambda^e$ where $\Lambda^e = \Lambda\otimes_k\Lambda^{op}$.
By restricting this correspondence to our setting, we have the following proposition.

\begin{Prop}\label{TriAICorrespondence}
    Let $A, B$ be as above. Let $M$ (resp. $N$) be the distinguished $d\mathbb{Z}$-cluster tilting module of $A$ (resp. $B$). 
    Then 
    \begin{equation*}
        D_{sg}(M,M) \cong \underline{\End}_B(N) = A_{n', \ell' - 1}^{(d+1)}.
    \end{equation*}
    Moreover:
    \begin{itemize}
        \item [(i)] $D_{sg}(A)$ has a unique dg-enhancement.
        \item [(ii)] Let $\mathcal{T}$ be an algebraic Krull-Schmidt triangulated category with finite dimensional morphism spaces.
        If there exists a basic $d\mathbb{Z}$-cluster tilting object $c\in \mathcal{T}$ such that $\mathcal{T}(c,c) \cong A_{n', \ell' - 1}^{(d+1)}$. 
        Then
        \begin{equation*}
            \mathcal{T} \simeq D_{sg}(A) 
        \end{equation*}
        as triangulated categories.
    \end{itemize}
\end{Prop}
\begin{proof}
    Combined with Corollary \ref{corSingEquiToStableCat}, this is parallel to \cite[Theorem 6.5.2]{JKM22}.
\end{proof}

Let $\Lambda = A_{n, \ell- 1}^{(d+1)}$ be a self-injective $(d+1)$-Nakayama algebra with $n\geq 1, \ell\geq 2$. 
Denote by $Q_{\Lambda}$ the Gabriel quiver of $\Lambda$. 
We define an automorphism $\Phi$ of $Q_{\Lambda}$ as follows.
    \begin{align*}
        \Phi: Q_{\Lambda} & \rightarrow Q_{\Lambda}\\
        (x_1, \ldots, x_{d+1}) & \mapsto (f(x_{d+1}), x_1, \ldots, x_d)\\
        [a_i(x): x\rightarrow x+e_i] & \mapsto [a_{i+1}(\Phi(x)): \Phi(x)\rightarrow \Phi(x) + e_{i+1}]
    \end{align*}
    where $f(i) = i - \ell - d + 1$ for $i\in \mathbb{Z}$. By convention, let $a_{d+2} = a_1$ and $e_{d+2} = e_1$.
This is well-defined since $x_1 \geq x_{d+1} - \ell - d + 2 > f(x_{d+1})$.
Moreover, the relations of $Q_{\Lambda}$ are invariant under $\Phi$ since 
\begin{align*}
    0 & = \Phi(a_j(x+e_i)a_i(x) - a_i(x+e_j)a_j(x))\\
    & = a_{j+1}(\Phi(x) + e_{i+1})a_{i+1}(\Phi(x)) - a_{i+1}(\Phi(x) + e_{j+1})a_{j+1}(\Phi(x)).
\end{align*}
 Hence we can extend $\Phi$ linearly to get an algebra automorphism $\Phi: \Lambda\xrightarrow[]{\sim} \Lambda$. 
 Additionally, we denote by $(-)_{\Phi}: \modn \Lambda \rightarrow \modn \Lambda$ the auto-equivalence induced by $\Phi$.

\begin{Prop}
    Let $\Lambda$ and $\Phi$ be as above. 
    Then $\Lambda$ is twisted $(d+2)$-periodic, that is, $\Omega_{\Lambda}^{d+2} \cong (\text{ } )_{\Phi}$ as functors in $\underline{\modn }\Lambda$.
\end{Prop}
\begin{proof}
    By Proposition \ref{TriAICorrespondence}, $\Lambda \cong \underline{\End}_{\Gamma}(M)$ where $\Gamma = A_{n, \ell}^{(d)}$ and $M$ is the distinguished $d\mathbb{Z}$-cluster tilting module of $\Gamma$.
    Let $H = \underline{\Hom}_{\Gamma}(M, -)$.

    We identify the Gabriel quiver $Q_{\Lambda}$ of $\Lambda$ with the Auslander-Reiten quiver of $\underline{\add M}$.
    By Proposition \ref{prop_f_detects_syzygyorbit},
    it follows that $\Omega_{\Gamma}^d(M(x)) = M(\Phi(x))$ with $M(x)\in \underline{\add M}$.
    
    Let $f_{yx}^i: M(x)\rightarrow M(y)$ be a nonzero morphism in $\underline{\add M}$.
    We claim that $H\Omega_{\Gamma}^d(f_{yx}^i) = (H(f_{yx}^i))_{\Phi}= f_{\Phi(y)\Phi(x)}^i$.
    Without loss of generality, 
    we may assume $i = 0$. 
    It follows that $f(y_{d+1}) < x_1$. 
    If not, then $x\preccurlyeq x' = (x_1, \ldots, x_d, x_1 + \ell + d - 1) \preccurlyeq y$.
    Hence $f_{yx}^0 = f_{yx'}^0f_{x'x}^0$.
    Note that $M(x')$ is projective.
    This contradicts with $f_{yx}^0$ being nonzero in $\underline{\add M}$. 
    
    Applying Proposition \ref{prop_f_detects_syzygyorbit} to $M(x), M(y)$ and lifting $f_{yx}^0$ we obtain
    \begin{equation*}
        \xymatrix@C=1.5em@R=1.5em{
        M(\Phi(x)) \ar[r]\ar[d]^{f_{\Phi(y)\Phi(x)}^0} & P^d\ar[r]\ar[d]^{g_d} & \cdots\ar[r] & P^1\ar[r]\ar[d]^{g_1} & M(x)\ar[r]\ar[d]^{f_{yx}^0} & 0\\
        M(\Phi(y)) \ar[r] & Q^d\ar[r] & \cdots\ar[r] & Q^1\ar[r] & M(y)\ar[r] & 0.
        }
    \end{equation*} 
    Observe that $g_i \neq 0$ for $1\leq i\leq d$ since $f(y_{d+1}) < x_1$ as shown above.
    As claimed $H\Omega_{\Gamma}^d(f_{yx}^i) \cong (H(f_{yx}^i))_{\Phi}$. Hence
    we have the following commutative diagram, for a more general statement, cf \cite[Proposition 2.2.7]{JKM22}.
    \begin{equation*}
        \xymatrix@C=5em@R=3em{
        \underline{\add M} \ar[r]^{H}\ar[d]^{\Omega_{\Gamma}^d} & \proj \Lambda\ar[d]^{(-)_{\Phi}}\\
        \underline{\add M}\ar[r]^{H} & \proj \Lambda
        }.
    \end{equation*}
    Denote by $\varepsilon:  H\Omega_{\Gamma}^d \xrightarrow[]{\sim} (-)_{\Phi}H$ the natural isomorphism.

    Now we show that $\Omega_{\Lambda}^{d+2} \cong (-)_{\Phi}$ on $\underline{\modn} \Lambda$. 
    
    Let $N$ be an indecomposable object in $\underline{\modn} \Lambda$ and take a minimal projective presentation of $N$ in $\modn \Lambda$
    \begin{equation*}
        P_1\xrightarrow[]{g} P_0\rightarrow N\rightarrow 0.
    \end{equation*}
     Then there exist $M_0, M_1\in \add M$ and $\beta: M_1\rightarrow M_0$ such that $P_i \cong HM_i$ for $i = 0, 1$ and $g = H\beta$.
    
    Since $\underline{\add M}$ is a $d\mathbb{Z}$-cluster tilting subcategory of $\underline{\modn} \Gamma$, it has a $(d+2)$-angulated structure. 
    Thus we embed $\beta$ into a $(d+2)$-angle 
    \begin{equation*}
        (\ast\ast) ~~\xymatrix@C=1em@R=1em{\Omega_{\Gamma}^dM_0\ar[r] & M_{d+1}\ar[r] & \cdots\ar[r] & M_2\ar[r] & M_1\ar[r]^{\beta} & M_0}.
    \end{equation*}
    Applying $H$ to $(\ast\ast)$, we have the following exact sequence
    \begin{equation*}
        \xymatrix@C=1em@R=1em{
        H(\Omega_{\Gamma}^{d}M_1)\ar[r]^{h} &  H(\Omega_{\Gamma}^{d}M_0)\ar[r] & 
        H(M_{d+1})\ar[r] &
        \cdots\ar[r] 
        & H(M_1)\ar[r]^{g} &
        H(M_0)\ar[r] &
        N\ar[r] & 0
        }.
    \end{equation*}
    Consider the following commutative diagram with exact rows.
    \begin{equation*}
        \xymatrix@C=1.8em@R=1.8em{
        H(\Omega_{\Gamma}^{d}M_1)\ar[r]^{h}\ar[d]^{\varepsilon_{M_1}} &  H(\Omega_{\Gamma}^{d}M_0)\ar[r]\ar[d]^{\varepsilon_{M_0}} & \Omega_{\Lambda}^{d+2}(N)\ar[r]\ar[d] & 0\\
        (H(M_1))_{\Phi}\ar[r]^{g_{\Phi}} &
        (H(M_0))_{\Phi}\ar[r] &
        N_{\Phi}\ar[r] & 0
        }
    \end{equation*}
    Since $\varepsilon_{M_1}$ and $\varepsilon_{M_0}$ are isomorphisms, 
    we have that $\Omega_{\Lambda}^{d+2}(N) \cong N_{\Phi}$.

    Let $\varphi : N\rightarrow N'$ be a morphism in $\underline{\modn}\Lambda$.
    We have the following diagram. 
    \begin{equation*}
        \xymatrix@!0@C=4em@R=3em{
        & H(\Omega_{\Gamma}^d M_1) \ar[rr]\ar'[d][dd]\ar[dl] && H(\Omega_{\Gamma}^d M_0) \ar[rr]\ar'[d][dd]\ar[dl] && \Omega_{\Lambda}^{d+2}(N) \ar'[d][dd]\ar[rr]\ar[dl] && 0 \\
        H(\Omega_{\Gamma}^d M'_1)\ar[rr]\ar[dd] && H(\Omega_{\Gamma}^d M'_0)\ar[dd]\ar[rr] && \Omega_{\Lambda}^{d+2}(N')\ar[dd]\ar[rr] && 0\\
        & (H(M_1))_{\Phi}\ar'[r][rr]\ar[dl] && (H(M_0))_{\Phi}\ar'[r][rr]\ar[dl] && N_{\Phi}\ar[rr]\ar[dl] && 0\\
        (H(M'_1))_{\Phi}\ar[rr] \ar[rr] && (H(M'_0))_{\Phi}\ar[rr] && N'_{\Phi}\ar[rr] && 0
        }
    \end{equation*}
    The rightmost vertical square commutes since all the other squares commute. This show that $\Omega_{\Lambda}^{d+2} \cong (-)_{\Phi}$ as functors on $\underline{\modn} \Lambda$.
    
\end{proof}

\begin{Eg}
    Let $n = 5$, $d = 2$ and $\underline{\ell} = (3,4,4,4,4)$. By Example \ref{egWideSubCat}, 
    the Gabriel quiver of $\Lambda = A_{4, 2}^{(3)}$ is given as follows.
    \begin{equation*}
    {\small
        \begin{xy}
            0;<30pt,0cm>:<10pt,20pt>::
            (0,0) *+{123} = "123",
            (0,1) *+{124} = "124",
            (1,1) *+{134} = "134",
            (2,0) *+{234} = "234",
            (2,1) *+{235} = "235",
            (3,1) *+{245} = "245",
            (4,0) *+{345} = "345",
            (4,1) *+{346} = "346",
            (5,1) *+{356} = "356",
            (6,0) *+{456} = "456",
            (6,1) *+{457} = "457",
            (7,1) *+{467} = "467",
            (8,0) *+{123} = "567",
            (8,1) *+{124} = "568",
            "123", {\ar"124"},
            "124", {\ar"134"},
            "234", {\ar"235"},
            "235", {\ar"245"},
            "345", {\ar"346"},
            "346", {\ar"356"},
            "456", {\ar"457"},
            "457", {\ar"467"},
            "567", {\ar"568"},
            "134", {\ar"234"},
            "245", {\ar"345"},
            "356", {\ar"456"},
            "467", {\ar"567"},
        \end{xy}}
    \end{equation*}
    $\Lambda$ is twisted $4$-periodic and the twist is induced by the automorphism $\Phi$ which sends $(x_1, x_2, x_3)$ to $(x_3-4, x_1, x_2)$.
\end{Eg}

\section{Examples}
In this section, we give more examples.

\begin{Eg}(Compare to \cite[Example 5.4]{Sh15})
    Let $n = 4$, $d = 1$ and $\underline{\ell} = (5,6,7,6)$.
    Let $A = A_{\underline{\ell}}^{(1)}$ be the usual Nakayama algebra and $\modn A$ is the $1\mathbb{Z}$-cluster tilting subcategory.
    The resolution quiver is given as follows.
    \begin{equation*}
        \xymatrix@R=1em@C=2em{
         1\ar[dr] \\
         & 4\ar@/^/[r] & 2\ar@/^/[l]\\
         3\ar[ur]
        }
    \end{equation*}
    Then $J = \{2, 4\}$ and $I = J + 4\mathbb{Z}$.
    Thus $\iota(1) = 2$ and $\iota(2) = 4$.
    Therefore $\ell' = |[f(\iota(1)), \iota(1)]\cap I| - 1 = |[-4, 2] \cap I| - 1 = 3$ and $B = A_{2, 3}^{(1)}$. 
    The Auslander-Reiten quiver of the wide subcateory $\mathcal{W}$ of $\modn A$ is as follows.
    \begin{equation*}
        {\small
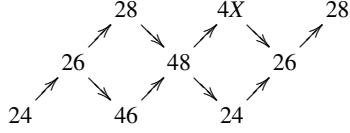
        \begin{xy}
            0;<20pt,0cm>:<20pt,20pt>::
            (0,0) *+{24} = "00",
            (0,1) *+{26} = "01",
            (0,2) *+{28} = "02",
            (2,0) *+{46} = "10",
            (2,1) *+{48} = "11",
            (2,2) *+{4X} = "12",
            (4,0) *+{24} = "20",
            (4,1) *+{26} = "21",
            (4,2) *+{28} = "22",
            "00", {\ar"01"},
            "01", {\ar"02"},
            "10", {\ar"11"},
            "11", {\ar"12"},
            "20", {\ar"21"},
            "21", {\ar"22"},
            "01", {\ar"10"},
            "02", {\ar"11"},
            "11", {\ar"20"},
            "12", {\ar"21"},
        \end{xy}
        }
    \end{equation*}
    Hence we have that $\Lambda = A_{2,2}^{(2)}$ which is twisted $3$-periodic.
\end{Eg}

\begin{Eg}
    Let $n = 5$, $d = 4$ and $\underline{\ell} = (5,5,6,6,5)$.
    Let $A = A_{\underline{\ell}}^{(4)}$ be the $4$-Nakayama algebra defined by $\underline{\ell}$ and $\mathcal{M}$ the distinguished $4\mathbb{Z}$-cluster tilting subcategory.
    Recall that $\overline{f}(i) \equiv i - \ell_i - 3 \mod 5$.
    Thus we have the resolution quiver
    \begin{equation*}
        \xymatrix@R=1em@C=1em{
            1\ar[r] & 3\ar[r] & 4\ar[r] & 5\ar[r] & 2\ar@/^1pc/[ll]}.
    \end{equation*}\medskip
    
    Then $J = \{2,4,5\}$ and $I = J + 5\mathbb{Z}$.
    We have the Auslander-Reiten quiver of the wide subcategory $\mathcal{W}$ of $\mathcal{M}$ as follows.
    \begin{equation*}
        {\tiny
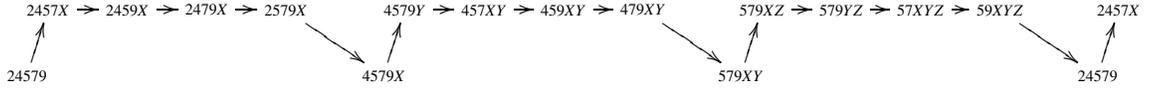
        \begin{xy}
            0;<30pt,0cm>:<8pt,25pt>::
            (0,0) *+{24579} = "1",
            (0,1) *+{2457X} = "2",
            (1,1) *+{2459X} = "3",
            (2,1) *+{2479X} = "4",
            (3,1) *+{2579X} = "5",
            (4.5,0) *+{4579X} = "6",
            (4.5,1) *+{4579Y} = "7",
            (5.5,1) *+{457XY} = "8",
            (6.5,1) *+{459XY} = "9",
            (7.5,1) *+{479XY} = "10",
            (9,0) *+{579XY} = "11",
            (9,1) *+{579XZ} = "12",
            (10,1) *+{579YZ} = "13",
            (11,1) *+{57XYZ} = "14",
            (12,1) *+{59XYZ} = "15",
            (13.5,0) *+{24579} = "16",
            (13.5,1) *+{2457X} = "17",
            "1", {\ar"2"},
            "2", {\ar"3"},
            "3", {\ar"4"},
            "4", {\ar"5"},
            "5", {\ar"6"},
            "6", {\ar"7"},
            "7", {\ar"8"},
            "8", {\ar"9"},
            "9", {\ar"10"},
            "10", {\ar"11"},
            "11", {\ar"12"},
            "12", {\ar"13"},
            "13", {\ar"14"},
            "14", {\ar"15"},
            "15", {\ar"16"},
            "16", {\ar "17"},
        \end{xy}
        }
    \end{equation*}
    Thus $B = A_{3, 2}^{(4)}$ and $\Lambda = k\oplus k\oplus k$. 
\end{Eg}

\section*{Acknowledgments}

The author would like to thank her advisor Martin Herschend for many helpful comments and discussions. The author also wants to thank Jordan McMahon for pointing out that the main theorem has been proved in his paper in 2019.

\bibliographystyle{amsplain}

\providecommand{\bysame}{\leavevmode\hbox to3em{\hrulefill}\thinspace}
\renewcommand{\MRhref}[2]{%
  \href{http://www.ams.org/mathscinet-getitem?mr=#1}{#2}
}
\renewcommand\MR[1]{\relax\ifhmode\unskip\space\fi MR~\MRhref{#1}{#1}}


\end{document}